\documentclass[12pt]{amsart}

\usepackage{amsmath}
\usepackage{amssymb}
\usepackage{amsfonts}
\usepackage{mathrsfs}
\usepackage{graphicx}
\usepackage{epsfig}
\usepackage{accents}
\usepackage{geometry}
\usepackage[pagewise]{lineno}
\usepackage{amscd,amsmath,amssymb,amsthm,amsfonts,epsfig,graphics}
\usepackage{fancyhdr}

\usepackage{tikz}
\usepackage{pgf,pgfplots}
\usepackage{caption}

\usetikzlibrary{calc}

\allowdisplaybreaks

\newtheorem{definition}{Definition}[section]

\theoremstyle{plain}
\newtheorem{thm}{Theorem}
\newtheorem{lem}{Lemma}[section]
\newtheorem{cor}[lem]{corollary}
\newtheorem{fact}{Fact}[section]
\newtheorem{prop}{Proposition}
\newtheorem{remark}{Remark}

\newcommand{\m}{\textrm{mod}}
\newcommand{\diam}{\textrm{diam}}
\newcommand{\dist}{\textrm{dist}}

\tikzset{
	my rectangle/.pic={
		\draw (0,-0.07) rectangle (0.8,0.07);
	}
}

\tikzset{
	my draw/.pic={
		\draw (0,0) -- (4,0);
		\draw (0,-0.1) -- (0,0.1);
		\draw (4,-0.1) -- (4,0.1);
		\draw[dashed] (4,0) -- (5.2,0);
		\draw (1.4,-0.07) rectangle (2.6,0.07);
		\fill (2,0) circle (0.07);
	}
}

\tikzset{
	my square/.pic={
		\draw (-2,-2) rectangle (2,2);
		\draw[dotted] (-2,-2) -- (2,2);
		\draw (-1.5,-2.1) -- (-1.5,-1.9);
		\fill (-0.75,-2)node[shift={(-90:10pt)}] {$c$} circle (1.5pt);
		\draw (0,-2.1)node[shift={(-90:7pt)}] {$p$} -- (0,-1.9);
		\fill (0.75,-2)node[shift={(-90:10pt)}] {$d$} circle (1.5pt);
		\draw (1.5,-2.1) -- (1.5,-1.9);
	}
}

\tikzset{
	my squ/.pic={
		\draw (-1.6,-1.6) rectangle (1.6,1.6);
		\draw[dotted] (-1.6,-1.6) -- (1.6,1.6);
		\draw[dotted] (-1.4,-1.6) -- (-1.4,1.4) -- (1.4,1.4) -- (1.4,-1.6);
		\draw[dotted] (-1.4,-1.4) -- (1.4,-1.4);
		\draw[dotted] (-0.2,1.4) -- (-0.2,-1.6);
		\draw[dotted] (0.2,1.4) -- (0.2,-1.6);
		\draw[dotted] (-1.4,0.2) -- (1.4,0.2);
		\draw[dotted] (-1.4,-0.2) -- (1.4,-0.2);
		\fill (-0.8,-1.6)node[shift={(-90:9pt)}] {$c$} circle (1pt);
		\fill (0.8,-1.6)node[shift={(-90:8pt)}] {$d$} circle (1pt);
	}
}

\tikzset{
	my arc11/.pic={
		\draw plot [smooth, tension=1] coordinates {(-0.2,0) (0,1) (0.2,0)};
	}
}

\tikzset{
	my arc12/.pic={
		\draw plot [smooth, tension=1] coordinates {(-0.2,0) (0,-1) (0.2,0)};
	}
}

\tikzset{
	my arc21/.pic={
		\clip (-0.6,0.2) rectangle (0,1.4);
		\draw plot [smooth, tension=1] coordinates {(-0.6,0) (0,1.8) (0.6,0)};
	}
}

\tikzset{
	my arc22/.pic={
		\clip (0,0.2) rectangle (0.6,1.4);
		\draw plot [smooth, tension=1] coordinates {(-0.6,0) (0,1.8) (0.6,0)};
	}
}

\tikzset{
	my arc23/.pic={
		\clip (-0.6,-0.2) rectangle (0,-1.4);
		\draw plot [smooth, tension=1] coordinates {(-0.6,0) (0,-1.8) (0.6,0)};
	}
}

\tikzset{
	my arc24/.pic={
		\clip (0,-0.2) rectangle (0.6,-1.4);
		\draw plot [smooth, tension=1] coordinates {(-0.6,0) (0,-1.8) (0.6,0)};
	}
}

\begin{document}

\title[]{Decay of geometry for a class of cubic polynomials}

\author{Haoyang Ji and Wenxiu Ma}

\address{School of Mathematics and Statistics,
Zhengzhou University, Zhengzhou,
450001,  CHINA (e-mail:jihymath@zzu.edu.cn)}
\address{School of Mathematical Sciences,
University of Science and Technology of China, Hefei,
230026,  CHINA (e-mail:mwx@mail.ustc.edu.cn)}

\date{}

\begin{abstract}
In this paper we study a class of bimodal cubic polynomials for which its critical points have the same $\omega$-limit set which is an invariant Cantor set. These maps have generalized Fibonacci combinatorics in terms of generalized renormalization on the twin principal nest.  It is proved that such maps possess `decay of geometry' in the sense that the scaling factor of the twin principal nest decreases at least exponentially fast. As an application, we prove that they have no Cantor attractor.
\end{abstract}

\maketitle

\section{Introduction}
\label{intro}

The dynamical properties of unimodal interval maps have been extensively studied. By now, we have reached a full understanding of real analytic unimodal dynamics, especially for real quadratic polynomials, or more broadly, $S$-unimodal maps with critical point of order 2. A phenomenon named ``(exponentially) decay of geometry" plays essential role in the study of quadratic dynamics. Several results including density of hyperbolicity and Milnor's attractor problem rely on this phenomenon.

Informally speaking, decay of geometry in unimodal case means that one can find sufficiently small nice neighborhoods $V_n \subset V_n'$ of the critical point so that the ratio $|V_n|/|V_n'|$ decreases at least exponentially fast. (Actually, the relation between $V_n$ and $V_n'$ should be more specific.) It was proved in \cite{JS} and \cite{Lyu} that for S-unimodal maps with critical order $\ell \leq 2$, the decay of geometry property follows from a ``starting condition". The verification of the starting condition is more complicated. It was proved by Lyubich\cite{Lyu} for quadratic maps. Lately, Graczyk-Sands-\'Swi\c atiek \cite{GSS} gave an alternative proof for S-unimodal maps with non-degenerate critical point, using the method of asymptotically conformal extension which goes back to Sullivan. Note that these proofs of the starting condition make elaborate use of complex methods and do not seem to work for critical order smaller than 2. More recently, Shen \cite{S} proved the decay of geometry property for smooth unimodal maps with critical order between 1 and 2 by a purely real argument.

However, the decay of geometry loses its universality for unimodal maps with larger critical order, or even for multimodal maps with quadratic critical points. In unimodal case, a famous example is the Fibonacci unimodal maps. Maps of Fibonacci combinatorics have already attracted a lot of attention. It is well-known at present that Fibonacci unimodal maps have decay of geometry for critical order $\ell \leq 2$ and bounded geometry for critical order $\ell > 2$, see \cite{LM} and \cite{KN}. Note that bounded geometry arose from the study of infinitely renormalizable maps with bounded combinatorics. On the other hand,  in \cite{SV} \'Swi\c atiek and Vargas provided a concrete example of a non-renormalizable bimodal cubic polynomial with bounded geometry. The construction of their example was based on the inducing procedure goes back to Jakobson. By using principal nest, Shen in \cite{S2} showed that multimodal map with non-degenerate critical points has either `large bounds' or `essentially bounded geometry'. In \cite{V} Vargas introduced the Fibonacci bimodal map by use of the natural symmetry of bimodal maps. Instead of considering one sequence of nice neighborhoods of a chosen critical point, Vargas constructed two sequences of both critical points simultaneously. We call the two sequences of nice intervals the {\it twin principal nest}.

The aim of this paper is to study decay of geometry phenomenon in the bimodal cubic polynomial case. We describe a class of bimodal cubic polynomials including cubic Fibonacci map by using the language of generalized renormalization.  The main result is that every map from this class has decay of geometry in the sense that the scaling factor of its twin principal nest decreases at least exponentially fast. The strategy is to extend the generalized renormalization to a generalized box mapping and then study the asymptotic property of the moduli of the fundamental annuli. Since the {\it real bounds} theorem does not hold in our settings, we will use the construction of Yoccoz puzzle to get {\it complex a priori bounds} first. The Yoccoz puzzle construction relies on B$\ddot{\rm o}$ttcher Coordinate which only holds for polynomials. This is the reason we state our result for bimodal cubic polynomials. Moreover, in our setting the non-existence of Cantor attractor holds.

\subsection{Preliminaries} Denote $I=[0,1]$. A continuous map $f: I \to I$ is called  {\it bimodal} if $f(\{0,1 \}) = \{0,1 \}$ and $f$ has exactly one local maximum and one local minimum in $(0,1)$. The two extreme points specified by $c <d$ are called {\it turning points} and $f$ is strictly monotone on subintervals determined by these points. If the points $\{0,1 \}$ are fixed then we say that the bimodal map $f$ is {\it positive} and in the case that these points are permuted we say that $f$ is {\it negative}. Examples of bimodal maps are parameterized families of real cubic polynomials $P_{ab}^+$ and $P_{ab}^-$ given by $P_{ab}^+(x) = a x^3 + b x^2 + (1-a-b)x$ and $P_{ab}^-(x) = 1 -a x^3 -b x^2 - (1-a-b)x$. We are mainly interested in bimodal maps have neither periodic attractors nor wandering intervals.

For $T\subset I$, let $D(T)=\{x\in I: f^k(x)\in T \mbox{ for some }k\geq 1\}$.
The {\it first entry map} $R_T: D(T)\to T$ is defined as $x \to f^{k(x)}(x)$, where $k(x)$ is the {\it entry time} of $x$ into $T$, i.e., the minimal positive integer such that $f^{k(x)}(x)\in T$. The map $R_T|(D(T)\cap T)$ is called the {\it first return map} of $T$. A component of $D(T)$  is called an {\it entry domain} of $T$ and a component of $D(T) \cap T$ is called a {\it return domain}.

A point $x \in I$ is called {\it recurrent} provided $x \in \omega(x)$.

If $J$ and $J'$ are two intervals on the real line, by $J<J'(J \leq J')$ we mean that $y < y' (y \leq y')$ for every $y \in J$ and $y' \in J'$; analogously we define $a < J$ and $a \leq J'$ for real number $a$.

\subsection{Twin principal nest} Let $\mathscr B$ denote the collection of $C^{3}$ bimodal maps $f : I \to I$ which have no wandering intervals and all periodic cycles hyperbolic repelling. Let $\textrm{Crit}(f)$ denote the set of critical points of $f$, i.e., the set of points where $Df$ vanishes. Note that $\{c ,d \} \subset \textrm{Crit}(f)$. Let $\mathscr B^+$ and $\mathscr B^-$ denote, respectively, the subset of positive and negative bimodal maps from class $\mathscr B$. If $f \in \mathscr B^+$, then there exists a fixed point $p$ in $(c, d)$; for otherwise $\partial I$ contains attracting fixed point. Let $p_1 < p_2$ be such that $f(p_1) = f(p_2) =p$. Define $I^0 = (p_1,p)$, $J^0 = (p , p_2)$. If $f \in \mathscr B^-$, we discuss in three cases:
\begin{itemize}
\item[(1)] $f$ has three fixed point in $(0,1)$. In this case, there exists a fixed point $p$ in $(c,d)$, then define $I^0$ and $J^0$ as above.
\item[(2)] $f$ has one fixed point $p$ in $(0,1)$ with three preimages $\{p,p_{1},p_{2}\}$ specified by $p_{1}< p_{2}$. If $p< p_{1}<p_{2}$, define $I^{0} = (p,p_1) \ni c$ and $J^{0} = (p_1 , p_2) \ni d$; if $p_1<p_2<p$, define $I^{0}=(p_{1},p_{2})$ and $J^{0} = (p_{2},p)$.
\item[(3)] $f$ has one fixed point $p$ in $(0 ,1)$ with only one preimage, that is, $f^{-1}(p) = \{p\}$. This case can be reduced to the positive case since $f^2$ restricted on $[p,1]$ is always a positive bimodal map.
\end{itemize}

 Assume that $c$ and $d$ are recurrent. Define
\[ I^0 \supset I^1 \supset I^2 \supset \ldots \supset \{c\} \ \mbox{and} \ J^0 \supset J^1 \supset J^2 \supset \ldots \supset \{d\}
\]
such that, for $n \geq 1$, the intervals $I^n$ and $J^n$ are first return domains to $I^{n-1} \cup J^{n-1}$. The two sequences of nested intervals will be called the {\it twin principal nest} of $f$. The {\it scaling factor} of $f$ is defined as $$\lambda_n : = \max\{\frac{|I^{n}|}{|I^{n-1}|} , \frac{|J^{n}|}{|J^{n-1}|}  \}.$$

Let $g_n$ denote the first return map to $I^{n-1} \cup J^{n-1}$. Let $J$ be any connected component of the domain of $g_n$, the restriction of $g_n$ on $J$ is called a {\it branch}. Then $g_n|J = g_{n-1}^{k}$, where $k = \min\{ k \geq 1 : g_{n-1}^k(x) \in I^{n-1} \cup J^{n-1}\}$. A branch $J$ of $g_n$ is called {\it post-critical} if $J$ contains $g_n(c)$ or $g_n(d)$; a branch $J$ of $g_n$ is called {\it immediate} if $g_n|J = g_{n-1}$. The first return map $g_n$ is called {\it central return} if $g_n(c) \in I^n \cup J^n $ or $g_n(d) \in I^n \cup J^n $; otherwise $g_n$ is called {\it non-central return}. In case that $g_n$ is non-central, $g_n$ may have 1 or 2 post-critical branches. When $g_n$ has 2 post-critical branches denoted as $C^n$ and $D^n$, we may assume that $C^n < D^n$, which means that $C^n$ is on the left side of $D^n$.

When a post-critical branch $J$ is also immediate, sometimes we will call $J$ a {\it Fibonacci branch}. So if $g_n$ has a Fibonacci branch, then the corresponding critical point leaves the central domains $I^n \cup J^n$ but then returns back immediately under $g_n$, this is the fastest combinatorial recurrence in the non-central case.

We shall now describe our example in terms of the first return map to twin principal nest.

\begin{definition}
Let $\mathscr G$ denote the class of bimodal maps $f \in \mathscr B$ with two recurrent critical points $c, d$ and such that the following properties hold:
\begin{itemize}
\item[(1)] $\omega(c) = \omega(d)$;
\item[(2)] $f(c),f(d) \notin I^0 \cup J^0$ while $f^i(c), f^i(d) \in I^0 \cup J^0$ for $i=2,3$;
\item[(3)] $C^n$ and $D^n$ are well-defined for all $n \geq 1$;
\item[(4)] for each $n \geq 1$, $I^n \cup J^n \subset g_n(I^n \cup J^n)$;
\item[(5)] for each $n \geq 1$, $(\omega(c) \cup \omega(d)) \cap (I^{n-1} \cup J^{n-1}) \subset I^n \cup J^n \cup C^n \cup D^n$;
\item[(6)] for each $n \geq 1$, $g_n|(I^n \cup J^n) = g_{n-1}^{r_n}|(I^n \cup J^n)$ for some integer $r_n \geq 2$;
\item[(7)] for each $n \geq 1$, $g_n|(C^n \cup D^n) = g_{n-1}^{t_n}|(C^n \cup D^n)$ for some integer $t_n \geq 1$.
\end{itemize}
\end{definition}

If $\omega(c) \neq \omega(d)$, the orbits of two critical points are separated, hence can be reduced to unimodal case. Property (2) should be considered as a starting condition. Property (3) and (4) allow that on each level $n$, $g_n$ is always non-central and high return type. While property (5) implies that on each level the number of first return domains intersecting the orbits of $c$ and $d$ are exactly 4, including 2 central domains and 2 post-critical domains. Property (6) and (7) refer to the renormalization type of a generalized renormalization.

\subsection{Box mappings} In the unimodal case, a generalized renormalization is defined as the first return map to $n$-th level principal nest restricted on return domains intersecting the orbits of the critical point. In our settings, the $n$-th generalized renormalization for $f \in \mathscr G$ is just the restriction of $g_n$ on the union of 4 return domains: $I^n \cup J^n \cup C^n \cup D^n$. In the following we shall introduce the notion  {\it real (and complex) box mapping}, which can be considered as a special type of generalized renormalization via twin principal nest. The definition below applies to both real and complex mappings, once we recognize that an interval is a topological disk on the real line and replace `monotone branch' with `univalent branch'.

\begin{definition}
Let $\mathcal F$ be the collection of smooth maps
\[
\phi : U_0 \cup U_1 \cup V_0 \cup V_1 \to U \cup V
\]
with the following properties:
\begin{itemize}
\item $U$ and $V$ are open intervals with pairwise disjoint closures specified by $U<V$;
\item the domains of $f$ are open intervals with pairwise disjoint closures such that $U_0 \cup U_1 \subset U$ and $V_0 \cup V_1 \subset V$;
\item $\phi|U_1$ and $\phi|V_1$ are monotone onto $U$ or $V$; 
\item $\phi|U_0$ has a unique critical point $c$ while $\phi|V_0$ has a unique critical point $d$;
\item $\phi$ is a proper map and extends continuously to the closure of its domain.
\end{itemize}
Moreover, we may make the following two additional assumptions:
\begin{itemize}
\item the critical points $c$ and $d$ do not escape under the forward iterates of $\phi$; 
\item the critical points $c$ and $d$ are non-periodic and recurrent, and they have the same $\omega$-limit set.
\end{itemize}
\end{definition}

For a box mapping $\phi$ from class $\mathcal F$, the concepts of nice open sets, first return maps, etc, can be formulated similarly as in the case as interval endomorphism. The restriction of a box mapping $\phi$ to connected components of its domain will be referred to as {\it branches} as well. In particular, $\phi|U_0$ and $\phi|V_0$ will be called central branch, while the other two will be called monotone branch.

We will see below that the combinatorial property of our class will make some constraints about the position of the return domains  and their images. In other words, we mainly consider the following 4 types of box mapping $\phi$:
\begin{itemize}
\item Type $\mathcal A$: if $\phi(U_1) = V$, $\phi(U_0) \subset V$ and $\phi(V_1) = U$, $\phi(V_0) \subset U$;
\item Type $\mathcal B$: if $\phi(U_1) = V$, $\phi(U_0) \subset U$ and $\phi(V_1) = U$, $\phi(V_0) \subset V$;
\item Type $\mathcal C$: if $\phi(U_1) = U$, $\phi(U_0) \subset V$ and $\phi(V_1) = V$, $\phi(V_0) \subset U$;
\item Type $\mathcal D$: if $\phi(U_1) = U$, $\phi(U_0) \subset U$ and $\phi(V_1) = V$, $\phi(V_0) \subset V$.
\end{itemize}
Observe that the precise position and orientation of the central and non-central branches are not specified yet. Figure 1 illustrate a possible position but still without the orientation.

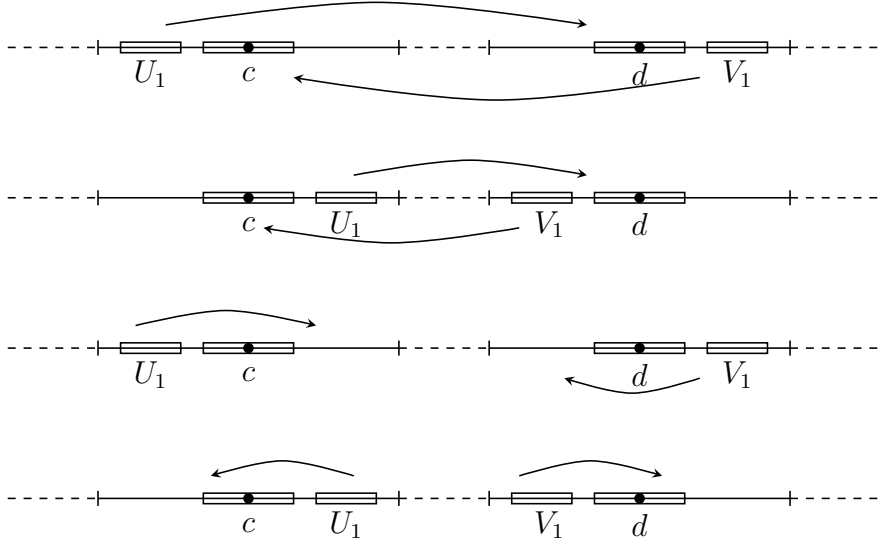
\begin{figure}
		\centering
		\begin{tikzpicture}[scale=1, line width=0.6pt, >=stealth]
			\draw[dashed] (0,-2) -- (1.2,-2);
			\pic at(1.2,-2) {my draw};
			\pic at(6.4,-2) {my draw};
			\node[shift={(-90:10pt)}] at(3.2,-2) {$c$};
			\node[shift={(-90:10pt)}] at(8.4,-2) {$d$};
			\pic at(6.7,-2) {my rectangle};
			\node[shift={(-90:10pt)}] at(7.2,-2) {$V_1$};
			\pic at(4.1,-2) {my rectangle};
			\node[shift={(-90:10pt)}] at(4.5,-2) {$U_1$};
			\draw[->] plot [smooth, tension=0.6] coordinates {(4.6,-1.7) (3.65,-1.5) (2.7,-1.7)};
			\draw[->] plot [smooth, tension=0.6] coordinates {(6.8,-1.7) (7.75,-1.5) (8.7,-1.7)};

			\draw[dashed] (0,0) -- (1.2,0);
			\pic at(1.2,0) {my draw};
			\pic at(6.4,0) {my draw};
			\node[shift={(-90:10pt)}] at(3.2,0) {$c$};
			\node[shift={(-90:10pt)}] at(8.4,0) {$d$};
			\pic at(1.5,0) {my rectangle};
			\node[shift={(-90:10pt)}] at(1.9,0) {$U_1$};
			\pic at(9.3,0) {my rectangle};
			\node[shift={(-90:10pt)}] at(9.7,0) {$V_1$};
			\draw[->] plot [smooth, tension=0.6] coordinates {(1.7,0.3) (2.9,0.5) (4.1,0.3)};
			\draw[->] plot [smooth, tension=0.6] coordinates {(9.2,-0.4) (8.3,-0.6) (7.4,-0.4)};
			
			\draw[dashed] (0,2) -- (1.2,2);
			\pic at(1.2,2) {my draw};
			\pic at(6.4,2) {my draw};
			\node[shift={(-90:10pt)}] at(3.2,2) {$c$};
			\node[shift={(-90:10pt)}] at(8.4,2) {$d$};
			\pic at(6.7,2) {my rectangle};
			\node[shift={(-90:10pt)}] at(7.2,2) {$V_1$};
			\pic at(4.1,2) {my rectangle};
			\node[shift={(-90:10pt)}] at(4.5,2) {$U_1$};
			\draw[->] plot [smooth, tension=0.6] coordinates {(4.6,2.3) (6.15,2.5) (7.7,2.3)};
			\draw[->] plot [smooth, tension=0.6] coordinates {(6.8,1.6) (5.1,1.4) (3.4,1.6)};
			
			\draw[dashed] (0,4) -- (1.2,4);
			\pic at(1.2,4) {my draw};
			\pic at(6.4,4) {my draw};
			\node[shift={(-90:10pt)}] at(3.2,4) {$c$};
			\node[shift={(-90:10pt)}] at(8.4,4) {$d$};
			\pic at(1.5,4) {my rectangle};
			\node[shift={(-90:10pt)}] at(1.9,4) {$U_1$};
			\pic at(9.3,4) {my rectangle};
			\node[shift={(-90:10pt)}] at(9.7,4) {$V_1$};
			\draw[->] plot [smooth, tension=0.6] coordinates {(2.1,4.3) (4.9,4.6) (7.7,4.3)};
			\draw[->] plot [smooth, tension=0.6] coordinates {(9.2,3.6) (6.5,3.3) (3.8,3.6)};
			
		\end{tikzpicture}
		\caption{Examples of types $\mathcal{A}$ $\mathcal{B}$ $\mathcal{C}$ $\mathcal{D}$}
	\end{figure}

A box mapping $\phi$ is called {\it symmetric} if it satisfies further:
\begin{itemize}
\item $\phi$ is of type $\mathcal A$, $\mathcal B$, $\mathcal C$ or $\mathcal D$;
\item $\phi$ has different type of local extreme at $c$ and $d$;
\item $\phi|U_1$ and $\phi|V_1$ have the same orientation.
\end{itemize}
Let $\mathcal F_s$ denote the collection of symmetric box mappings. Let $\mathcal F_s^+$ and $\mathcal F_s^-$, respectively, denote the collection of maps from $\mathcal F_s$ such that the monotone branches of $\phi$ are orientation-preserving and orientation-reversing.

We will divide the type of symmetric box mappings in more details. Let us devide each type $\mathcal A$, $\mathcal B$, $\mathcal C$, $\mathcal D$ into subtypes $\mathcal A^{ij}$, $\mathcal B^{ij}$, $\mathcal C^{ij}$, $\mathcal D^{ij}$ with $i,j \in \{ +,- \}$. Where $i=+$ or $i=-$ if the monotone branches of $\phi$ are orientation-preserving or orientation-reversing; and $j=+$ or $j=-$ if $\phi$ is local maximal or minimal at $c$. Let $\mathcal T = \{\mathcal A^{ij}, \mathcal B^{ij}, \mathcal C^{ij}, \mathcal D^{ij}  \}$. Let $\mathcal A^+ : = \mathcal A^{++} \cup \mathcal A^{+-}$ and define $\mathcal A^-, \mathcal B^+, \mathcal B^-, \mathcal C^+, \mathcal C^-, \mathcal D^+, \mathcal D^-$ analogously.

\subsection{Statement of Results} We will see in section 2 that the $n$-th generalized renormalization $g_n$ is always a symmetric box mapping. Therefore, for each $n \geq 1$, $g_n$ is associated with 3 terms: its type, denoted as $\theta_n \in \mathcal T$, and return times $r_n$ and $t_n$ under the iterate of $g_{n-1}$. Then the original map $f \in \mathscr G$ is associated with a sequence of triples $\mathcal S(f) = \{(\theta_n, r_n, t_n)\}_{n \geq 1}$ which will be called the {\it combinatorial sequence} of $f$. We will show that $\mathcal S(f)$ satisfies some additional condition, named the {\it Admissible condition}. We remark here that the combinatorial sequence $\mathcal S(f)$ always starts with $\mathcal A^{++}$ or $\mathcal C^{-+}$ depending on $f$ is positive or negative. On the other hand, let $\mathcal S$ be any sequence of triples $\{(\theta_n, r_n, t_n)\}_{n \geq 1}$ satisfying $\theta_n \in \mathcal T$, $r_n \geq 2, t_n \geq 1$ for all $n \geq 1$. It is proved that if $\mathcal S$ satisfies the Admissible condition, then $\mathcal S$ is {\it admissible}, which means there exists a bimodal map $f \in \mathscr G$ such that $\mathcal S(f) = \mathcal S$. Hence the Admissible condition is necessary and also sufficient.

According to \cite{MS}, the families of real cubic polynomials $P_{ab}^+$ and $P_{ab}^-$ are full families. Combining with the Rigidity Theorem developed in \cite{KSS}, we can obtain the following corollary.

\begin{cor}
Let $\mathcal S$ be any sequence satisfying the Admissible condition. 
\begin{itemize}
\item If $\theta_1 = \mathcal A^{++}$, there exists a unique map in $P_{ab}^+$ with combinatorial sequence $\mathcal S$;
\item If $\theta_1 = \mathcal C^{-+}$, there exists a unique map in $P_{ab}^-$ with combinatorial sequence $\mathcal S$.
\end{itemize} 
\end{cor}

\begin{definition}
Let $\mathscr C$ denote $\mathscr G \cap (P_{ab}^+ \cup P_{ab}^-)$.
\end{definition}

The main result of this paper is the following:

\begin{thm}
Suppose $f \in \mathscr C$, then there exist constants $C = C(f) > 0$ and $0<\lambda=\lambda(f) <1$ such that the scaling factor of $f$ decreases at least exponentially: $\lambda_n(f) \leq C \lambda^n$ for all $n \geq 1$.
\end{thm}

This property is known as `decay of geometry' for unimodal maps. It is well-known that the non-existence of `wild attractor' is based only on the decay of geometry in the unimodal case. In our setting we prove the following proposition.

\begin{prop}
Suppose $f \in \mathscr C$, then $f$ admits no Cantor attractor.
\end{prop}

This paper is organized as follows. In section 2, we study the combinatorial properties of maps in class $\mathscr G$. In section 3, by using Yoccoz puzzle method we show that the first return map $g_n$ can be treated as a complex box mapping. In subsection 4.1  we summarize basic background about conformal roughness from \cite{GS}. In subsection 4.2 we introduce the separation index for complex box mappings in our setting. In subsection 4.3 we prove the {\it complex a priori bounds} by showing that the separation index is non-decreasing. The proof of Theorem 1 is given in subsection 4.4, the idea is to show that the separation index grows at a linear rate. In subsection 4.5 we prove Proposition 1 by a standard random walk argument.

\section{Combinatorics}

\subsection{Combinatorial type} From the description of class $\mathscr G$, any map $f \in \mathscr G$ can be considered as a special real box mapping which is infinitely renormalizable in the context of generalized renormalization. Furthermore, Property (6) and (7) in Definition 1.1 will refer to the combinatorial type of a generalized renormalization.

Suppose $\phi \in \mathcal F_s$. If there exists a positive integer $r>1$ such that $\phi(c), \ldots , \phi^{r-1}(c) \in  U_1 \cup V_1$ while $ \phi^r(c) \in U_0 \cup V_0$, then we say that $\phi$ makes a return with depth $r$ at $c$. In this case, the first return map to $U_0 \cup V_0$ is well-defined near $c$ and equals $\phi^r$.

Suppose $\phi$ makes a return with depth $r$ at both $c$ and $d$. Let $U_0^1$ and $V_0^1$ be the connected components of $D(U_0 \cup V_0)$ which contains $c$ and $d$ respectively. In case that $\phi^r(c),\phi^r(d) \notin U_0^1 \cup V_0^1$, we say that $\phi$ makes non-central return; otherwise, we say that $\phi$ makes central return. (This is the same as in subsection 1.2.) When $\phi$ makes non-central return, let $U_1^1 \subset U_1$ and $V_1^1 \subset V_1$ denote the first return domains contains $\phi^r(c)$ and $\phi^r(d)$ respectively. Since $\phi^r(c)$ and $\phi^r(d)$ are in different components of $U_0 \cup V_0$, $U_1^1$ and $V_1^1$ do not coincide.

\begin{definition}
Suppose $\phi \in \mathcal F_s$, we say that $\phi$ has combinatorial type $(r, t)$, where $r \geq 2$ and $t \geq 1$, if
\begin{itemize}
\item[(1)] $\phi(U_0 \cup V_0) \supset U_0 \cup V_0$;
\item[(2)] f makes a non-central return with depth $r$ at $c$ and $d$;
\item[(3)] the first return time of $U_1^1$ and $V_1^1$ to $U_0 \cup V_0$ equals $t$.
\end{itemize}
A new map {\bf induced} from $\phi$, denoted $\mathcal I \phi$, is the first return map $R_{U_0 \cup V_0}$ restricted on $U_0^1 \cup U_1^1 \cup V_0^1 \cup V_1^1$.
\end{definition}

Property (1) here corresponds to property (4) in Definition 1.1. Property (2) and (3) correspond to property (6) and (7) in Definition 1.1. It is clear that $\mathcal I \phi$ equals $\phi^r$ on $U_0^1 \cup V_0^1$ and $\phi^t$ on $U_1^1 \cup V_1^1$. Moreover, if $\phi$ satisfies property (1), then the precise positions of all its branches are clear. Figure 2 below illustrates the graphs of $\phi$ in all possible cases.

\begin{figure}[hb!]
		\centering
		\begin{tikzpicture}[scale=1, line width=0.6pt, >=stealth]
			\pic at(-6,6) {my squ};
			\pic at($(-6,6)+(-0.8,0)$) {my arc21};
			\pic at($(-6,6)+(-0.8,0.2)$) {my arc11};
			\pic at($(-6,6)+(0.8,-0.2)$) {my arc12};
			\pic at($(-6,6)+(0.8,0)$) {my arc24};
			\draw[dotted] ($(-6,6)+(-0.8,-1.6)$) -- ($(-6,6)+(-0.8,1.2)$);
			\draw[dotted] ($(-6,6)+(0.8,-1.6)$) -- ($(-6,6)+(0.8,-1.2)$);
			
			\pic at(-2,6) {my squ};
			\pic at($(-2,6)+(-0.8,1.4)$) {my arc12};
			\pic at($(-2,6)+(-0.8,1.6)$) {my arc24};
			\pic at($(-2,6)+(0.8,-1.6)$) {my arc21};
			\pic at($(-2,6)+(0.8,-1.4)$) {my arc11};
			\draw[dotted] ($(-2,6)+(-0.8,-1.6)$) -- ($(-2,6)+(-0.8,0.4)$);
			\draw[dotted] ($(-2,6)+(0.8,-1.6)$) -- ($(-2,6)+(0.8,-0.4)$);
			
			\pic at(2,6) {my squ};
			\pic at($(2,6)+(-0.8,1.6)$) {my arc23};
			\pic at($(2,6)+(-0.8,0.2)$) {my arc11};
			\pic at($(2,6)+(0.8,-0.2)$) {my arc12};
			\pic at($(2,6)+(0.8,-1.6)$) {my arc22};
			\draw[dotted] ($(2,6)+(-0.8,-1.6)$) -- ($(2,6)+(-0.8,1.2)$);
			\draw[dotted] ($(2,6)+(0.8,-1.6)$) -- ($(2,6)+(0.8,-1.2)$);
			
			\pic at(6,6) {my squ};
			\pic at($(6,6)+(-0.8,1.4)$) {my arc12};
			\pic at($(6,6)+(-0.8,0)$) {my arc22};
			\pic at($(6,6)+(0.8,0)$) {my arc23};
			\pic at($(6,6)+(0.8,-1.4)$) {my arc11};
			\draw[dotted] ($(6,6)+(-0.8,-1.6)$) -- ($(6,6)+(-0.8,0.4)$);
			\draw[dotted] ($(6,6)+(0.8,-1.6)$) -- ($(6,6)+(0.8,-0.4)$);
			
			\pic at(-6,2) {my squ};
			\pic at($(-6,2)+(-0.8,-1.4)$) {my arc11};
			\pic at($(-6,2)+(-0.8,1.6)$) {my arc24};
			\pic at($(-6,2)+(0.8,-1.6)$) {my arc21};
			\pic at($(-6,2)+(0.8,1.4)$) {my arc12};
			\draw[dotted] ($(-6,2)+(-0.8,-1.6)$) -- ($(-6,2)+(-0.8,-0.4)$);
			\draw[dotted] ($(-6,2)+(0.8,-1.6)$) -- ($(-6,2)+(0.8,0.4)$);
			
			\pic at(-2,2) {my squ};
			\pic at($(-2,2)+(-0.8,0)$) {my arc21};
			\pic at($(-2,2)+(-0.8,-0.2)$) {my arc12};
			\pic at($(-2,2)+(0.8,0.2)$) {my arc11};
			\pic at($(-2,2)+(0.8,0)$) {my arc24};
			\draw[dotted] ($(-2,2)+(-0.8,-1.6)$) -- ($(-2,2)+(-0.8,-1.2)$);
			\draw[dotted] ($(-2,2)+(0.8,-1.6)$) -- ($(-2,2)+(0.8,1.2)$);
			
			\pic at(2,2) {my squ};
			\pic at($(2,2)+(-0.8,-1.4)$) {my arc11};
			\pic at($(2,2)+(-0.8,0)$) {my arc22};
			\pic at($(2,2)+(0.8,0)$) {my arc23};
			\pic at($(2,2)+(0.8,1.4)$) {my arc12};
			\draw[dotted] ($(2,2)+(-0.8,-1.6)$) -- ($(2,2)+(-0.8,-0.4)$);
			\draw[dotted] ($(2,2)+(0.8,-1.6)$) -- ($(2,2)+(0.8,0.4)$);
			
			\pic at(6,2) {my squ};
			\pic at($(6,2)+(-0.8,1.6)$) {my arc23};
			\pic at($(6,2)+(-0.8,-0.2)$) {my arc12};
			\pic at($(6,2)+(0.8,0.2)$) {my arc11};
			\pic at($(6,2)+(0.8,-1.6)$) {my arc22};
			\draw[dotted] ($(6,2)+(-0.8,-1.6)$) -- ($(6,2)+(-0.8,-1.2)$);
			\draw[dotted] ($(6,2)+(0.8,-1.6)$) -- ($(6,2)+(0.8,1.2)$);
			
			\pic at(-6,-2) {my squ};
			\pic at($(-6,-2)+(-0.8,-1.6)$) {my arc21};
			\pic at($(-6,-2)+(-0.8,0.2)$) {my arc11};
			\pic at($(-6,-2)+(0.8,-0.2)$) {my arc12};
			\pic at($(-6,-2)+(0.8,1.6)$) {my arc24};
			\draw[dotted] ($(-6,-2)+(-0.8,-1.6)$) -- ($(-6,-2)+(-0.8,1.2)$);
			\draw[dotted] ($(-6,-2)+(0.8,-1.6)$) -- ($(-6,-2)+(0.8,-1.2)$);
			
			\pic at(-2,-2) {my squ};
			\pic at($(-2,-2)+(-0.8,1.4)$) {my arc12};
			\pic at($(-2,-2)+(-0.8,0)$) {my arc24};
			\pic at($(-2,-2)+(0.8,0)$) {my arc21};
			\pic at($(-2,-2)+(0.8,-1.4)$) {my arc11};
			\draw[dotted] ($(-2,-2)+(-0.8,-1.6)$) -- ($(-2,-2)+(-0.8,0.4)$);
			\draw[dotted] ($(-2,-2)+(0.8,-1.6)$) -- ($(-2,-2)+(0.8,-0.4)$);
			
			\pic at(2,-2) {my squ};
			\pic at($(2,-2)+(-0.8,0)$) {my arc23};
			\pic at($(2,-2)+(-0.8,0.2)$) {my arc11};
			\pic at($(2,-2)+(0.8,-0.2)$) {my arc12};
			\pic at($(2,-2)+(0.8,0)$) {my arc22};
			\draw[dotted] ($(2,-2)+(-0.8,-1.6)$) -- ($(2,-2)+(-0.8,1.2)$);
			\draw[dotted] ($(2,-2)+(0.8,-1.6)$) -- ($(2,-2)+(0.8,-1.2)$);
			
			\pic at(6,-2) {my squ};
			\pic at($(6,-2)+(-0.8,1.4)$) {my arc12};
			\pic at($(6,-2)+(-0.8,-1.6)$) {my arc22};
			\pic at($(6,-2)+(0.8,1.6)$) {my arc23};
			\pic at($(6,-2)+(0.8,-1.4)$) {my arc11};
			\draw[dotted] ($(6,-2)+(-0.8,-1.6)$) -- ($(6,-2)+(-0.8,0.4)$);
			\draw[dotted] ($(6,-2)+(0.8,-1.6)$) -- ($(6,-2)+(0.8,-0.4)$);
			
			\pic at(-6,-6) {my squ};
			\pic at($(-6,-6)+(-0.8,-1.4)$) {my arc11};
			\pic at($(-6,-6)+(-0.8,0)$) {my arc24};
			\pic at($(-6,-6)+(0.8,0)$) {my arc21};
			\pic at($(-6,-6)+(0.8,1.4)$) {my arc12};
			\draw[dotted] ($(-6,-6)+(-0.8,-1.6)$) -- ($(-6,-6)+(-0.8,-0.4)$);
			\draw[dotted] ($(-6,-6)+(0.8,-1.6)$) -- ($(-6,-6)+(0.8,0.4)$);
			
			\pic at(6,-6) {my squ};
			\pic at($(6,-6)+(-0.8,0)$) {my arc23};
			\pic at($(6,-6)+(-0.8,-0.2)$) {my arc12};
			\pic at($(6,-6)+(0.8,0.2)$) {my arc11};
			\pic at($(6,-6)+(0.8,0)$) {my arc22};
			\draw[dotted] ($(6,-6)+(-0.8,-1.6)$) -- ($(6,-6)+(-0.8,-1.2)$);
			\draw[dotted] ($(6,-6)+(0.8,-1.6)$) -- ($(6,-6)+(0.8,1.2)$);
			
			\pic at(2,-6) {my squ};
			\pic at($(2,-6)+(-0.8,-1.4)$) {my arc11};
			\pic at($(2,-6)+(-0.8,-1.6)$) {my arc22};
			\pic at($(2,-6)+(0.8,1.6)$) {my arc23};
			\pic at($(2,-6)+(0.8,1.4)$) {my arc12};
			\draw[dotted] ($(2,-6)+(-0.8,-1.6)$) -- ($(2,-6)+(-0.8,-0.4)$);
			\draw[dotted] ($(2,-6)+(0.8,-1.6)$) -- ($(2,-6)+(0.8,0.4)$);
			
			\pic at(-2,-6) {my squ};
			\pic at($(-2,-6)+(-0.8,-1.6)$) {my arc21};
			\pic at($(-2,-6)+(-0.8,-0.2)$) {my arc12};
			\pic at($(-2,-6)+(0.8,0.2)$) {my arc11};
			\pic at($(-2,-6)+(0.8,1.6)$) {my arc24};
			\draw[dotted] ($(-2,-6)+(-0.8,-1.6)$) -- ($(-2,-6)+(-0.8,-1.2)$);
			\draw[dotted] ($(-2,-6)+(0.8,-1.6)$) -- ($(-2,-6)+(0.8,1.2)$);
		\end{tikzpicture}
		\caption{Types of $\mathcal A^{++}$, $\mathcal A^{+-}$, $\mathcal A^{-+}$, $\mathcal A^{--}$, $\mathcal B^{++}$, $\cdots$, $\mathcal D^{--}$}
	\end{figure}
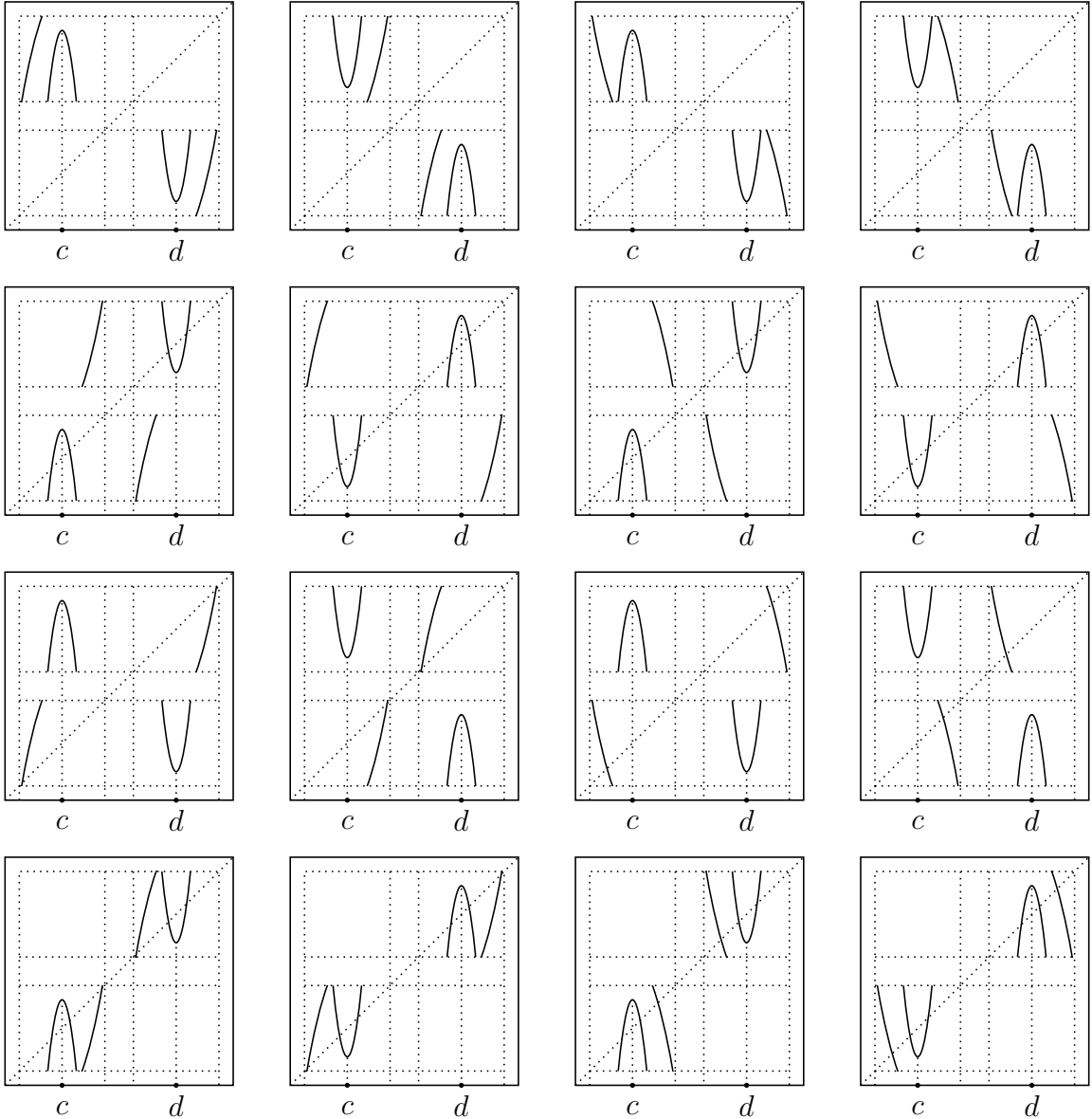

\subsection{Inducing} Consider the following ordering on the set of natural numbers, called the {\it admissible ordering}:
\[1 \prec 3 \prec 5 \prec 7 \prec \ldots \prec 2n+1 \prec \ldots \prec 2n \prec \ldots \prec 6 \prec 4 \prec 2.
\]
The following two lemmas show that the combinatorial type is related to the orientation on the monotone branches. Set up so that $\phi : U_0 \cup U_1 \cup V_0 \cup V_1  \to U \cup V$ and $\mathcal I \phi : U_0^1 \cup U_1^1 \cup V_0^1 \cup V_1^1 \to U_0 \cup V_0$.

\begin{lem}
Suppose that $\phi \in \mathcal F_s^-$ has combinatorial type $(r, t)$ with $r \geq 2$, $t \geq 1$,
\begin{itemize}
\item[(1)] if $\phi$ is of type $\mathcal A$ and $\mathcal B$, then $t<r$;
\item[(2)] if $\phi$ is of type $\mathcal C$, then $t \prec r$ in the admissible ordering.
\end{itemize}
\end{lem}

\begin{proof}
If $t=1$, obviously $t < r$. So it suffices to prove for $t \geq 2$. We argue as \cite[Lemma 2.4]{JL}.

(1) Without loss of generality, we may assume that $\phi$ is of type $\mathcal A^{-+}$, the other cases are similar. Then $\phi : U_1 \to V$ and $\phi : V_1 \to U$ are decreasing while $\phi|U_0$ is local maximal at $c$ and $\phi|V_0$ is local minimal at $d$. Since $\phi : U_0 \to V$ and $\phi(c) \in V_1$, $V_1$ is on the right side of $V_0$. Similarly, $U_1$ is on the left side of $U_0$. Since $\phi(U_1) \supset V_1$ and $\phi(V_1) \supset U_1$, there exists a periodic point $\alpha \in V_1$ with period 2. For $n \geq 1$, define
\begin{equation*}\label{eqn:kn}
K_n = \begin{cases}
\phi^{-n}(U_0) \cap V_1  & \text{ if $n$ is odd, }
\\
\phi^{-n}(V_0) \cap V_1 & \text{ if $n$ is even. }
\end{cases}
\end{equation*}
Then $K_n$ lie in $V_1$ as $K_1<K_2<K_3<\ldots < K_n <\ldots <\alpha$. Since $\phi|U_0$ is local maximal at c, $\phi(U_1^1) < \phi(U_0^1)$. By continuity, $\phi(U_0^1) \subset K_{r-1}$ and $\phi(U_1^1) \subset K_{t-1}$. This implies $t < r$.

(2) Since $\phi$ is of type $\mathcal C$, $\phi(U_1) \supset U_1$ while $\phi(V_1) \supset V_1$. Assume that $\phi$ is of type $\mathcal C^{-+}$. Let $K_n = (\phi|V_1)^{-n}(V_0)$ for $n \geq 1$. Then $K_n$ lie in $V_1$ as $K_2< K_4< \ldots < K_{2n} < \ldots < \beta < \ldots < K_{2n+1} < \ldots < K_3 < K_1$ where $\beta$ is a fixed point. Argue as above, we have $t \prec r$.
\end{proof}

\begin{lem}
Suppose that $\phi \in \mathcal F_s^+$ has combinatorial type $(r, t)$ with $r \geq 2$, $t \geq 1$,
\begin{itemize}
\item[(1)] if $\phi$ is of type $\mathcal A$ and $\mathcal B$, then $t \prec r$ in the admissible ordering;
\item[(2)] if $\phi$ is of type $\mathcal C$, then $t < r$.
\end{itemize}
\end{lem}

\begin{proof}
It suffices to prove for $t \geq 2$.

(1) We may assume that $\phi$ is of type $\mathcal A^{++}$. Then $\phi:U_1 \to V$ and $\phi: V_1 \to U$ are increasing while $\phi|U_0$ is local maximal at $c$ and $\phi|V_0$ is local minimal at $d$. For $n \geq 1$, define
\begin{equation*}\label{eqn:kn}
K_n = \begin{cases}
\phi^{-n}(U_0) \cap V_1  & \text{ if $n$ is odd, }
\\
\phi^{-n}(V_0) \cap V_1 & \text{ if $n$ is even. }
\end{cases}
\end{equation*}
Then $K_n$ lie in $V_1$ as $K_2< K_4< \ldots < K_{2n} < \ldots < \alpha < \ldots < K_{2n+1} < \ldots < K_3 < K_1$ where $\alpha$ is a periodic point with period 2. This implies $t \prec r$.

(2) Assume that $\phi$ is of type $\mathcal C^{++}$. Then $\phi:U_1 \to U$ and $\phi: V_1 \to V$ are increasing while $\phi|U_0$ is local maximal at $c$ and $\phi|V_0$ is local minimal at $d$. Since $\phi(V_1) \supset V_1$, define $K_n = (\phi|V_1)^{-n}(V_0)$ for $n \geq 1$. Then $K_n$ lie in $V_1$ as $K_1<K_2<K_3<\ldots < K_n <\ldots <\alpha$ where $\alpha$ is a fixed point. The same argument shows $t < r$.
\end{proof}

\begin{lem}
Suppose that $\phi \in \mathcal F_s$ has combinatorial type $(r, t)$ with $r \geq 2$, $t \geq 1$. Assume that $\mathcal I \phi (U_0^1 \cup V_0^1) \supset U_0^1 \cup V_0^1$, then $\mathcal I \phi \in \mathcal F_s$.
\end{lem}

\begin{proof}
Assume that $\phi$ is of type $\mathcal A^{-+}$ and $t < r$ both are odd by Lemma 2.1. The other cases are similar.  Then $\phi : U_1 \to V$ and $\phi : V_1 \to U$ are decreasing while $\phi|U_0$ is local maximal at $c$ and $\phi|V_0$ is local minimal at $d$. Moreover, $U_1 < U_0$ and $V_0 < V_1$.

Since $r$ is odd, $\mathcal I \phi | U_0^1 = \phi^r |U_0^1 = (\phi|U_1 \circ \phi|V_1)^{(r-1)/2} \circ (\phi|U_0^1)$ is local maximal at $c$ and $\phi^r(U_0^1) \subset V_0$. Since $\mathcal I \phi(U_0^1) \supset V_0^1$, $V_1^1$ is on the right side of $V_0^1$. Then $\phi|V_1^1$ is increasing. Since $t$ is odd, $\phi^t|V_1^1 = (\phi|V_1 \circ \phi|U_1)^{(t-1)/2} \circ (\phi|V_1^1)$ is increasing with $\phi^t(V_1^1) = U_0$. Similarly we have $\phi^r | V_0^1$ is local minimal at $d$ and $\phi^r(V_0^1) \subset U_0$. And also $\phi^t|U_1^1$ is increasing with $\phi^t(U_1^1) = V_0$. Hence $\mathcal I \phi$ belongs to $\mathcal F_s$. In particular, we have shown that $\mathcal I \phi$ is of type $\mathcal A^{++}$.

\end{proof}

The next two lemmas clarify all the possible cases for the inducing step, the proof is similar with Lemma 2.3. Given any integer $r, t \geq 1$, denote $e(r, t) = ((-1)^r, (-1)^t)$.

\begin{lem}
Suppose that $\phi \in \mathcal F_s^-$ has combinatorial type $(r, t)$ with $r \geq 2$, $t \geq 1$. Assume that $\mathcal I \phi (U_0^1 \cup V_0^1) \supset U_0^1 \cup V_0^1$, then we have the following:
\begin{itemize}
\item[(1)] if $\phi$ is of type $\mathcal A^{-+}$ with $t < r$,
\begin{itemize}
\item[(1.1)] if $e(r, t) = (-1, -1)$, then $\mathcal I \phi$ is of type $\mathcal A^{++}$;
\item[(1.2)] if $e(r, t) = (1, -1)$, then $\mathcal I \phi$ is of type $\mathcal B^{+-}$;
\item[(1.3)] if $e(r, t) = (-1, 1)$, then $\mathcal I \phi$ is of type $\mathcal C^{-+}$;
\item[(1.4)] if $e(r, t) = (1, 1)$, then $\mathcal I \phi$ is of type $\mathcal D$.
\end{itemize}
\item[(2)] if $\phi$ is of type $\mathcal A^{--}$ with $t < r$,
\begin{itemize}
\item[(2.1)] if $e(r, t) = (-1, -1)$, then $\mathcal I \phi$ is of type $\mathcal A^{+-}$;
\item[(2.2)] if $e(r, t) = (1, -1)$, then $\mathcal I \phi$ is of type $\mathcal B^{++}$;
\item[(2.3)] if $e(r, t) = (-1, 1)$, then $\mathcal I \phi$ is of type $\mathcal C^{--}$;
\item[(2.4)] if $e(r, t) = (1, 1)$, then $\mathcal I \phi$ is of type $\mathcal D$.
\end{itemize}
\item[(3)] if $\phi$ is of type $\mathcal B^{-+}$ with $t < r$,
\begin{itemize}
\item[(3.1)] if $e(r, t) = (-1, -1)$, then $\mathcal I \phi$ is of type $\mathcal D$;
\item[(3.2)] if $e(r, t) = (1, -1)$, then $\mathcal I \phi$ is of type $\mathcal C^{--}$;
\item[(3.3)] if $e(r, t) = (-1, 1)$, then $\mathcal I \phi$ is of type $\mathcal B^{++}$;
\item[(3.4)] if $e(r, t) = (1, 1)$, then $\mathcal I \phi$ is of type $\mathcal A^{+-}$.
\end{itemize}
\item[(4)] if $\phi$ is of type $\mathcal B^{--}$ with $t < r$,
\begin{itemize}
\item[(4.1)] if $e(r, t) = (-1, -1)$, then $\mathcal I \phi$ is of type $\mathcal D$;
\item[(4.2)] if $e(r, t) = (1, -1)$, then $\mathcal I \phi$ is of type $\mathcal C^{-+}$;
\item[(4.3)] if $e(r, t) = (-1, 1)$, then $\mathcal I \phi$ is of type $\mathcal B^{+-}$;
\item[(4.4)] if $e(r, t) = (1, 1)$, then $\mathcal I \phi$ is of type $\mathcal A^{++}$.
\end{itemize}
\item[(5)] if $\phi$ is of type $\mathcal C^{-+}$ with $t \prec r$,
\begin{itemize}
\item[(5.1)] if $e(r, t) = (-1, -1)$ then $\mathcal I \phi$ is of type $\mathcal A^{++}$;
\item[(5.2)] if $e(r, t) = (1, -1)$, then $\mathcal I \phi$ is of type $\mathcal A^{--}$;
\item[(5.3)] if $e(r, t) = (1, 1)$, then $\mathcal I \phi$ is of type $\mathcal A^{+-}$.
\end{itemize}
\item[(6)] if $\phi$ is of type $\mathcal C^{--}$ with $t \prec r$,
\begin{itemize}
\item[(6.1)] if $e(r, t) = (-1, -1)$, then $\mathcal I \phi$ is of type $\mathcal A^{+-}$;
\item[(6.2)] if $e(r, t) = (1, -1)$, then $\mathcal I \phi$ is of type $\mathcal A^{-+}$;
\item[(6.3)] if $e(r, t) = (1, 1)$, then $\mathcal I \phi$ is of type $\mathcal A^{++}$.
\end{itemize}
\end{itemize}
\end{lem}

\begin{lem}
Suppose that $\phi \in \mathcal F_s^+$ has combinatorial type $(r, t)$ with $r \geq 2$, $t \geq 1$. Assume that $\mathcal I \phi (U_0^1 \cup V_0^1) \supset U_0^1 \cup V_0^1$, then we have the following:
\begin{itemize}
\item[(1)] if $\phi$ is of type $\mathcal A^{++}$ with $t \prec r$,
\begin{itemize}
\item[(1.1)] if $e(r, t) = (-1, -1)$, then $\mathcal I \phi$ is of type $\mathcal A^{++}$;
\item[(1.2)] if $e(r, t) = (1, -1)$, then $\mathcal I \phi$ is of type $\mathcal B^{-+}$;
\item[(1.3)] if $e(r, t) = (1, 1)$, then $\mathcal I \phi$ is of type $\mathcal D$.
\end{itemize}
\item[(2)] if $\phi$ is of type $\mathcal A^{+-}$ with $t \prec r$,
\begin{itemize}
\item[(2.1)] if $e(r, t) = (-1, -1)$, then $\mathcal I \phi$ is of type $\mathcal A^{+-}$;
\item[(2.2)] if $e(r, t) = (1, -1)$, then $\mathcal I \phi$ is of type $\mathcal B^{--}$;
\item[(2.3)] if $e(r, t) = (1, 1)$, then $\mathcal I \phi$ is of type $\mathcal D$.
\end{itemize}
\item[(3)] if $\phi$ is of type $\mathcal B^{++}$ with $t \prec r$,
\begin{itemize}
\item[(3.1)] if $e(r, t) = (-1, -1)$, then $\mathcal I \phi$ is of type $\mathcal D$;
\item[(3.2)] if $e(r, t) = (1, -1)$ then $\mathcal I \phi$ is of type $\mathcal C^{++}$;
\item[(3.3)] if $e(r, t) = (1, 1)$, then $\mathcal I \phi$ is of type $\mathcal A^{++}$.
\end{itemize}
\item[(4)] if $\phi$ is of type $\mathcal B^{+-}$ with $t \prec r$,
\begin{itemize}
\item[(4.1)] if $e(r, t) = (-1, -1)$, then $\mathcal I \phi$ is of type $\mathcal D$;
\item[(4.2)] if $e(r, t) = (1, -1)$, then $\mathcal I \phi$ is of type $\mathcal C^{+-}$;
\item[(4.3)] if $e(r, t) = (1, 1)$, then $\mathcal I \phi$ is of type $\mathcal A^{+-}$.
\end{itemize}
\item[(5)] if $\phi$ is of type $\mathcal C^{++}$ with $t < r$, then $\mathcal I \phi$ is of type $\mathcal A^{++}.$
\item[(6)] if $\phi$ is of type $\mathcal C^{+-}$ with $t < r$, then $\mathcal I \phi$ is of type $\mathcal A^{+-}.$
\end{itemize}
\end{lem}

\subsection{Admissible condition} In this subsection we state and prove the Admissible condition. The following several lemmas will be useful.

\begin{lem}
Suppose $f \in \mathscr G$, then there exists a fixed point $p \in (c,d)$.
\end{lem}

\begin{proof}
It suffices to prove this lemma when $f$ is negative. Suppose that $f$ has one fixed point $p$ with three preimages and $p \notin (c,d)$. We may assume that $p \in (d,1)$. Let $p_1,p_2 \in f^{-1}(p)$ be such that $p_1 < p_2 <p$. Then $f|(p_1,c)$ and $f|(d,p)$ are decreasing. By property (2) and (3) in Definition 1.1, $f(c) < p_1 <c$ and $f(d) > p > d$. Then $f((c,d))\supset (c,d)$, hence contains a fixed point. A contradiction.
\end{proof}

\begin{lem}
Suppose $f \in \mathscr G$,  
\begin{itemize}
\item if $f$ is positive, then $g_1$ is of type $\mathcal A^{++}$; 
\item if $f$ is negative, then $g_1$ is of type $\mathcal C^{-+}$.
\end{itemize}
\end{lem}

\begin{proof}
We may assume that $f \in \mathscr B^+$. Then $f$ has a fixed point $p \in (c,d)$ which has three preimages specified by $p_1 < p < p_2$. Let $I^0=(p_1,p)$, $J^0 = (p ,p_2)$. By property (2) in Definition 1.1, $f(c), f(d) \notin I^0 \cup J^0$, then $f(c) \in (p_2,1)$ and $f(d) \in (0, p_1)$. Since $f|(p_2,1)$ is increasing, $f^2(c) >p$. Since $f^2(c), f^3(c) \in I^0 \cup J^0$, then $f^2(c) \in J^0$ and $f^3(c) \in I^0$. Therefore $g_1|I^1 = f^2$ with $g_1 : I^1 \to J^0$ is local maximal at $c$. Since $g_1(I^1) \supset J^1$, $D^1$ is on the right side of $d$. Then $g_1|D^1$ is increasing with $g_1(D^1) =f(D^1) = I^0$. Similarly, we have $g_1 : J^1 \to I^0$ is local minimal at $d$ and $g_1(C^1) = J^0$ is increasing. This shows that $g_1 : C^1 \cup I^1 \cup J^1 \cup D^1  \to I^0 \cup J^0 $ is of type $\mathcal A^{++}$. The case for $f$ negative is similar.
\end{proof}

\begin{figure}
		\centering
		\begin{tikzpicture}[scale=1, line width=0.6pt, >=stealth]
			\pic at(0,0) {my square};
			\draw[dotted] (1.5,0) -- (1.5,-1.5) -- (0,-1.5) -- (0,1.5) -- (-1.5,1.5) -- (-1.5,0) -- cycle;
			\draw plot [smooth, tension=0.6] coordinates {(-2,-2) (-0.75,1.8) (0.75,-1.8) (2,2)};
			\draw[dotted] (-0.75,-2) -- (-0.75,1.8);
			\draw[dotted] (0.75,-2) -- (0.75,-1.8);
			\node at(0.6,1.2) {$f$};
			
			\pic at(6,0) {my square};
			\draw[dotted] (7.5,0) -- (7.5,-1.5) -- (6,-1.5) -- (6,1.5) -- (4.5,1.5) -- (4.5,0) -- cycle;
			\draw plot [smooth, tension=1.2] coordinates {(5,0) (5.25,1.1) (5.5,0)};
			\draw plot [smooth, tension=1.2] coordinates {(6.5,0) (6.75,-1.1) (7,0)};
			\draw[dotted] (5.25,-2) -- (5.25,1.1);
			\draw[dotted] (6.75,-2) -- (6.75,-1.1);
			\node at(6.6,1.2) {$g_{1}$};
			\clip (4.5,0) rectangle (5.25,1.5) (6.75,-1.5) rectangle (7.5,0);
			\draw plot [smooth, tension=0.6] coordinates {(4.5,0) (5.25,1.8) (6.75,-1.8) (7.5,0)};
		\end{tikzpicture}
		\caption{graph of $f\in\mathscr B^+$ and $g_{1}$ of type $\mathcal{A}^{++}$}
	
		\centering
		\begin{tikzpicture}[scale=1, line width=0.6pt, >=stealth]
			\pic at(0,0) {my square};
			\draw[dotted] (1.5,0) -- (1.5,1.5) -- (0,1.5) -- (0,-1.5) -- (-1.5,-1.5) -- (-1.5,0) -- cycle;
			\draw plot [smooth, tension=0.6] coordinates {(-2,2) (-0.75,-1.8) (0.75,1.8) (2,-2)};
			\draw[dotted] (-0.75,-2) -- (-0.75,-1.8);
			\draw[dotted] (0.75,-2) -- (0.75,1.8);
			\node at(-1.3,1.2) {$f$};
			
			\pic at(6,0) {my square};
			\draw[dotted] (7.5,0) -- (7.5,1.5) -- (6,1.5) -- (6,-1.5) -- (4.5,-1.5) -- (4.5,0) -- cycle;
			\draw plot [smooth, tension=1.2] coordinates {(5,0) (5.25,1.1) (5.5,0)};
			\draw plot [smooth, tension=1.2] coordinates {(6.5,0) (6.75,-1.1) (7,0)};
			\draw[dotted] (5.25,-2) -- (5.25,1.1);
			\draw[dotted] (6.75,-2) -- (6.75,-1.1);
			\node at(4.7,1.2) {$g_{1}$};
			\clip (4.5,0) rectangle (5.25,-1.5) (6.75,1.5) rectangle (7.5,0);
			\draw plot [smooth, tension=0.6] coordinates {(4.5,0) (5.25,-1.8) (6.75,1.8) (7.5,0)};
		\end{tikzpicture}
		\caption{graph of $f\in\mathscr B^-$ and $g_{1}$ of type $\mathcal{C}^{-+}$}
	\end{figure}

\begin{cor}
Suppose $f \in \mathscr G$, then the $n$-th generalized renormalization $g_n$ is symmetric box mapping. Moreover, $g_n$ has combinatorial type $(r_n, t_n)$.
\end{cor}

\begin{proof}
By Lemma 2.7, $g_1$ is of type $\mathcal A^{++}$ or type $\mathcal C^{-+}$. By property (4), (6) and (7) in Definition 1.1, $g_1$ has combinatorial type $(r_1, t_1)$. By Lemma 2.3, $g_2$ is symmetric box mapping. Now this corollary follows by induction. 
\end{proof}

Corollary 2.8 shows that each generalized renormalization $g_n$ restricted on $C^n \cup I^n \cup J^n \cup D^n$ is always a symmetric box mapping with combinatorial type $(r_n, t_n)$. Hence the combinatorial sequence $\mathcal S = \{(\theta_n, r_n, t_n)\}_{n \geq 1}$ is well-defined. The next lemma implies that type $\mathcal D$ should be neglected.

\begin{lem}
Suppose $f \in \mathscr G$ with combinatorial sequence $\mathcal S = \{(\theta_n, r_n, t_n)\}_{n \geq 1}$. Then $\theta_n \notin \{ \mathcal D^{ij}\}$.
\end{lem}

\begin{proof}
Argue by contradiction.  Assume that $g_k$ is of type $\mathcal D$ for some $k \geq 1$, $g_k : I^k \cup C^k \to I^{k-1}$. Assume that $g_k|I^k = f^{S_k}$ and $g_k|C^k = f^{\hat S_k}$. Since $I^k$ is the first return domain, $f^i(I^k) \cap (I^{k-1} \cup J^{k-1}) = \emptyset $ for $1 \leq i \leq S_k-1$. Similarly $f^i(C^k) \cap (I^{k-1} \cup J^{k-1}) = \emptyset $ for $1 \leq i \leq \hat S_k-1$. By Definition 1.1 property (5), $\omega(c) \cap (I^{k-1} \cup J^{k-1}) \subset I^k \cup C^k \cup J^k \cup D^k$. Since $g_k(c) \in C^k$, we can show that $O_k : =[\cup_{i=0}^{S_k-1} f^i(\overline {I^k})] \cup [\cup_{i=0}^{\hat S_k-1} f^i(\overline {C^k})]$ is an cover of $\omega(c)$. But $O_k \cap J^{k-1} = \emptyset$. This implies $d \notin \omega(c)$, a contradiction.
\end{proof}

\noindent
{\bf Admissible condition.} Given any sequence of triples $\mathcal S = \{(\theta_n, r_n, t_n)\}_{n \geq 1}$ with $\theta_n \in \mathcal T$, $r_n \geq 2$, $t_n \geq 1$ for all $n \geq 1$. Then $\mathcal S$ satisfies the Admissible condition if the following inductive conditions holds: for all $n \geq 1$,
\begin{itemize}
\item[(1)] $\theta_1 = \mathcal A^{++} \ \mbox{or} \ \mathcal C^{-+}$ with $t_1 \prec r_1$.
\item[(2)] if $\theta_n \in \mathcal A^+$ with $t_n \prec r_n$, then $e(r_n, t_n) \neq (1,1)$.
\begin{itemize}
\item if $\theta_n = \mathcal A^{++}$ with $e(r_n, t_n) = (-1, -1)$, then $\theta_{n+1} = \mathcal A^{++}$ with $t_{n+1} \prec r_{n+1}$;
\item if $\theta_n = \mathcal A^{+-}$ with $e(r_n, t_n) = (-1, -1)$, then $\theta_{n+1} = \mathcal A^{+-}$ with $t_{n+1} \prec r_{n+1}$;
\item if $\theta_n = \mathcal A^{++}$ with $e(r_n, t_n) = (1, -1)$, then $\theta_{n+1} = \mathcal B^{-+}$ with $t_{n+1} < r_{n+1}$;
\item if $\theta_n = \mathcal A^{+-}$ with $e(r_n, t_n) = (1, -1)$, then $\theta_{n+1} = \mathcal B^{--}$ with $t_{n+1} < r_{n+1}$.
\end{itemize}
\item[(3)] if $\theta_n \in \mathcal A^-$ with $t_n < r_n$, then $e(r_n, t_n) \neq (1, 1)$.
\begin{itemize}
\item if $\theta_n = \mathcal A^{-+}$ with $e(r_n, t_n) = (-1, -1)$, then $\theta_{n+1} = \mathcal A^{++}$ with $t_{n+1} \prec r_{n+1}$;
\item if $\theta_n = \mathcal A^{--}$ with $e(r_n, t_n) = (-1, -1)$, then $\theta_{n+1} = \mathcal A^{+-}$ with $t_{n+1} \prec r_{n+1}$;
\item if $\theta_n = \mathcal A^{-+}$ with $e(r_n, t_n) = (-1, 1)$, then $\theta_{n+1} = \mathcal C^{-+}$ with $t_{n+1} \prec r_{n+1}$;
\item if $\theta_n = \mathcal A^{--}$ with $e(r_n, t_n) = (-1, 1)$, then $\theta_{n+1} = \mathcal C^{--}$ with $t_{n+1} \prec r_{n+1}$;
\item if $\theta_n = \mathcal A^{-+}$ with $e(r_n, t_n) = (1, -1)$, then $\theta_{n+1} = \mathcal B^{+-}$ with $t_{n+1} \prec r_{n+1}$;
\item if $\theta_n = \mathcal A^{--}$ with $e(r_n, t_n) = (1, -1)$, then $\theta_{n+1} = \mathcal B^{++}$ with $t_{n+1} \prec r_{n+1}$.
\end{itemize}
\item [(4)] if $\theta_n \in \mathcal B^+$ with $t_n \prec r_n$, then $e(r_n, t_n) \neq (-1, -1)$.
\begin{itemize}
\item if $\theta_n = \mathcal B^{++}$ with $e(r_n, t_n) = (1, -1)$, then $\theta_{n+1} = \mathcal C^{++}$ with $t_{n+1} < r_{n+1}$;
\item if $\theta_n = \mathcal B^{+-}$ with $e(r_n, t_n) = (1, -1)$, then $\theta_{n+1} = \mathcal C^{+-}$ with $t_{n+1} < r_{n+1}$;
\item if $\theta_n = \mathcal B^{++}$ with $e(r_n, t_n) = (1, 1)$, then $\theta_{n+1} = \mathcal A^{++}$ with $t_{n+1} \prec r_{n+1}$;
\item if $\theta_n = \mathcal B^{+-}$ with $e(r_n, t_n) = (1, 1)$, then $\theta_{n+1} = \mathcal A^{+-}$ with $t_{n+1} \prec r_{n+1}$.
\end{itemize}
\item[(5)] if $\theta_n \in \mathcal B^-$ with $t_n < r_n$, then $e(r_n, t_n) \neq (-1, -1)$.
\begin{itemize}
\item if $\theta_n = \mathcal B^{-+}$ with $e(r_n, t_n) = (-1, 1)$, then $\theta_{n+1} = \mathcal B^{++}$ with $t_{n+1} \prec r_{n+1}$;
\item if $\theta_n = \mathcal B^{--}$ with $e(r_n, t_n) = (-1, 1)$, then $\theta_{n+1} = \mathcal B^{+-}$ with $t_{n+1} \prec r_{n+1}$;
\item if $\theta_n = \mathcal B^{-+}$ with $e(r_n, t_n) = (1, -1)$, then $\theta_{n+1} = \mathcal C^{--}$ with $t_{n+1} \prec r_{n+1}$;
\item if $\theta_n = \mathcal B^{--}$ with $e(r_n, t_n) = (1, -1)$, then $\theta_{n+1} = \mathcal C^{-+}$ with $t_{n+1} \prec r_{n+1}$;
\item if $\theta_n = \mathcal B^{-+}$ with $e(r_n, t_n) = (1, 1)$, then $\theta_{n+1} = \mathcal A^{+-}$ with $t_{n+1} \prec r_{n+1}$;
\item if $\theta_n = \mathcal B^{--}$ with $e(r_n, t_n) = (1, 1)$, then $\theta_{n+1} = \mathcal A^{++}$ with $t_{n+1} \prec r_{n+1}$.
\end{itemize}
\item[(6)] if $\theta_n \in \mathcal C^+$ with $t_n < r_n$.
\begin{itemize}
\item if $\theta_n =\mathcal C^{++}$, then $\theta_n = \mathcal A^{++}$ with $t_{n+1} \prec r_{n+1}$;
\item if $\theta_n =\mathcal C^{+-}$, then $\theta_n = \mathcal A^{+-}$ with $t_{n+1} \prec r_{n+1}$
\end{itemize}
\item[(7)] if $\theta_n \in \mathcal C^-$ with $t_n \prec r_n$.
\begin{itemize}
\item if $\theta_n =\mathcal C^{-+}$ with $e(r_n, t_n) = (-1, -1)$, then $\theta_{n+1} = \mathcal A^{++}$ with $t_{n+1} \prec r_{n+1}$;
\item if $\theta_n =\mathcal C^{--}$ with $e(r_n, t_n) = (-1, -1)$, then $\theta_{n+1} = \mathcal A^{+-}$ with $t_{n+1} \prec r_{n+1}$
\item if $\theta_n =\mathcal C^{-+}$ with $e(r_n, t_n) = (1, -1)$, then $\theta_{n+1} = \mathcal A^{--}$ with $t_{n+1} < r_{n+1}$;
\item if $\theta_n =\mathcal C^{--}$ with $e(r_n, t_n) = (1, -1)$, then $\theta_{n+1} = \mathcal A^{-+}$ with $t_{n+1} < r_{n+1}$;
\item if $\theta_n =\mathcal C^{-+}$ with $e(r_n, t_n) = (1, 1)$, then $\theta_{n+1} = \mathcal A^{+-}$ with $t_{n+1} \prec r_{n+1}$;
\item if $\theta_n =\mathcal C^{--}$ with $e(r_n, t_n) = (1, 1)$, then $\theta_{n+1} = \mathcal A^{++}$ with $t_{n+1} \prec r_{n+1}$;
\end{itemize}
\end{itemize}

\begin{prop}
$\mathcal S$ is admissible if and only if $\mathcal S$ satisfies the Admissible condition.
\end{prop}

\begin{proof}
\

\noindent
{\bf Necessity.}
 Suppose $f \in \mathscr G$ has combinatorial sequence $\mathcal S(f) = \{ (\theta_n, r_n, t_n)\}_{n \geq 1}$. By induction, the $n$-th generalized renormalization $g_n$ is induced from $g_1$ by $n-1$ inducing step, i.e. $g_n = \mathcal I^{n-1} g_1$. By corollary 2.8 and Lemma 2.9, $\mathcal I^{n-1} g_1$ has combinatorial type $(r_n, t_n)$ and is of type $\mathcal A, \mathcal B, \mathcal C$. By Lemma 2.7, $g_1$ is of type $\mathcal A^{++}$ or $\mathcal C^{-+}$. Then by Lemma 2.1 and Lemma 2.2, in either case, we have $t_1 \prec r_1$. Now the necessity follows from Lemma 2.1-2.5 by induction.

\noindent
{\bf Sufficiency.} Let $\mathcal S = \{ (\theta_n, r_n, t_n)\}_{n \geq 1}$ be any sequence satisfying the Admissible condition. It suffices to construct a continuous bimodal map with the properties (2)-(7) in Definition 1.1. In fact, since $P_{ab}^+$ and $P_{ab}^-$ are full families, the existence of the corresponding map of the form follows. See \cite{MS}[Section II. 4]

Without loss of generality, we may assume that $\theta_1 = \mathcal A^{++}$ since the case $\theta_1 = \mathcal C^{-+}$ follows the same. To construct such a continuous bimodal map, we use a standard argument in interval dynamics. The proof is similar with \cite[Theorem 1]{JL}, so we mainly state the strategy here.

We shall first construct inductively a sequence of bimodal maps $f_n: I\to I$, $n=0,1,\ldots$, with the same turning points $c$ and $d$, such that the intervals $I^i$, $C^i$, $J^i$ and $D^i$ are well-defined for $1 \leq i \leq n+1$ and such that the properties (2)-(4) and (6)-(7) hold for $f=f_n$ and $k\le n$. These objects depend of course on $f_n$, and when we want to emphasize the map $f_n$, we write $I^i_{f_n}, C^i_{f_n}, g_n^{f_n}$, etc.

For the starting step, take $f_0$ be an arbitrary positive bimodal map with turning point $c,d\in (0,1)$ which has exactly one fixed point p between c and d. Moreover, $f_0$ satisfies
\[ f_0(d) < p_1 < f_0^2(d) < c =f_0^3(c) < p < d = f_0^3(d) < f_0^2(c) < p_2 < f_0(c),
\]
where $p_1, p_2 \in f_0^{-1}(p)$. This can be done by the classic kneading theory since the turning points are periodic. It is straightforward to check that for this map $f_0$, the return time of $c$ to $J^0_{f_0} = (p, p_2)$ is equal to 2 and the return time of $f_0^2 (d)$ to $J^0_{f_0}$ is 1. So $I^1_{f_0}$ and $C^1_{f_0}$ are well-defined and disjoint. Moreover, the map $f_0^2$ is local maximal at $c$, thus $f_0^2(I^1_{f_0}) \supset J^1_{f_0} \ni d$. Properties (6) and (7) are null in this case.

For the induction step, assume that $f_n$ is defined such that the properties (2)-(4) and (6)-(7) hold for $f=f_n$ and $k \leq n$, we shall modify $f_n$ near the turning points $c$ and $d$. To be precise, let us assume that the first return map $g_{n+1}^{f_n}$ restricted on $I^{n+1} \cup C^{n+1} \cup J^{n+1} \cup D^{n+1}$ is of type $\mathcal A^{++}$ and $t_{n+1} < r_{n+1}$ both are odd. Then $g_{n+1}^{f_n} : D^{n+1} \to I^n$ and $g_{n+1}^{f_n} : C^{n+1} \to J^n$ are monotone and onto. As in the proof of Lemma 2.2, let 
\[ R_1 = (g_{n+1}^{f_n}|C^{n+1} \cup D^{n+1})^{-(r_{n+1}-1)}(J^{n+1}) \cap D^{n+1},
\]
\[
R_2 = (g_{n+1}^{f_n}|C^{n+1} \cup D^{n+1})^{-(t_{n+1}-1)}(J^{n+1}) \cap D^{n+1}.
\]
Then $R_1 \cap R_2 = \emptyset$ and $R_1 \cup R_2 \subset D^{n+1}$. Since $t_{n+1} < r_{n+1}$, $R_2 < R_1$. Now modify $f_n$ inside $I^{n+1}$ so that $g_{n+1}^{f_n}(c) \in R_1$. Then $g_{n+1}^{f_n}(I^{n+1}) \supset R_2$. Let $C^{n+2}$ be the component of $(g_{n+1}^{f_n})^{-1}(R_2)$ which is on the left side of $c$. Modify $f_n$ inside $J^{n+1}$ to get $D^{n+2}$. Since $D^{n+2} \subset J^{n+1}$, it has a preimage $R' \subset R_1$ under $(g_{n+1}^{f_n})^{r_{n+1}-1}$. Then modify $f_n$ deep inside $I^{n+1}$ so that $g_{n+1}^{f_n}(c) \in R'$. The preimage of $R_1$ under $g_{n+1}^{f_n}|I^{n+1}$ (after modification) will be denoted $I^{n+2}$. Modify $f_n$ deep inside $J^{n+1}$ to get $J^{n+2}$. With this modification we can choose $f_{n+1}$ satisfying the induction assumption. Observe that $f_{n+1} = f_n$ outside $I^{n+1} \cup J^{n+1}$ for all $n$.

So these bimodal maps have been constructed. Moreover, from the construction we can choose $f_n$ so that $\{ f_n\}$ is equi-continuous and $|I^n|, |J^n| \to_n 0$. Thus $f_n$ converges to a continuous bimodal map $f$. Since $\theta_n \notin \{ \mathcal D^{ij}\}$ for all $n \geq 1$, we can find a sequence $\{ k_n\}_{n \geq 1}$ so that $f_n^{k_n}(c)$ enters $J^{n}$. By continuity, $f^{k_n}(c)$ enters the closure of $J^{n}$. Since $|J^n| \to 0$, $d \in \omega(c)$. Similarly, $c \in \omega(d)$. Hence $c$ and $d$ are recurrent with $\omega(c) = \omega(d)$. Moreover, $f^i(c)$ and $f^j(d)$ are disjoint from the boundary of $I^n$ and $J^n$ for all $n \geq 0$. By continuity, $f$ satisfies properties (2)-(4) and (6)-(7). Finally, we can prove property (5) by induction. This finishes the proof.

\end{proof}

Let $S_1 =2$, $\hat S_1=1$ and for each $n \geq 1$, define inductively
\[ S_{n+1} = S_n + (r_n-1) \hat S_n \mbox{ and } \hat S_{n+1} = S_n + (t_n-1) \hat S_n.
\]
The return times of $c$ and $d$ to $I^{n-1} \cup J^{n-1}$ are equal to $S_{n}$, while the return times of $g_n(c)$ and $g_n(d)$ to $I^{n-1} \cup J^{n-1}$ are equal to $\hat S_n$.

\begin{lem}
Suppose $f \in \mathscr G$, then $\omega(c) = \omega(d)$ is an invariant minimal Cantor set.
\end{lem}

\begin{proof}

For $k \geq 1$, consider the union of compact intervals in the following form:
\[ \Lambda_k : = [\cup_{i=0}^{S_k-1} f^i (\overline{I^k} \cup \overline{J^k})] \cup  [\cup_{i=0}^{\hat S_k-1} f^i (\overline{C^k} \cup \overline{D^k})].
\]
We can prove by induction that $\omega(c) \subset \Lambda_k$ for all $k \geq 1$. Let $\Lambda: = \cap_{k=1}^{\infty} \Lambda_k$. Then $\Lambda$ is a Cantor set. The minimality of $\Lambda$ follows from the fact that $f$ has no wandering intervals. Since $\omega(c) \subset \Lambda$, they must be coincident. 
\end{proof}

We shall provide several examples. We say that $f \in \mathscr G$ has stationary combinatorics $(r, t)$ provided $r_n = r$ and $t_n = t$ for fixed $r \geq 2, t \geq 1$ and for all $n$. This is the simplest combinatorics which corresponds to the fixed point of the generalized renormalization operator. The Admissible condition implies the fact: if $f$ has stationary combinatorics $(r, t)$, then $r$ and $t$ must be one of the following case:
\begin{itemize}
\item[(1)] $r$ is even and $t$ is odd with $t < r$;
\item[(2)] $r$ and $t$ both are odd with $t < r$.
\end{itemize}
For example $(5, 3)$ and $(4, 3)$ are admissible stationary combinatorics but $(3, 2)$ and $(3, 4)$ are not. Moreover, Lemma 2.4 and Lemma 2.5 show that
\begin{itemize}
\item[(3)] if $r$ is even and $t$ is odd, then the generalized renormalization sequence $\{g_n\}$ exhibits as
\[
\mathcal A^{++} \mathcal B^{-+} \mathcal C^{--} \mathcal A^{-+} \mathcal B^{+-} \mathcal C^{+-} \mathcal A^{+-} \mathcal B^{--} \mathcal C^{-+} \mathcal A^{--} \mathcal B^{++} \mathcal C^{++} \mathcal A^{++} \ldots
\]
or 
\[
\mathcal C^{-+} \mathcal A^{--} \mathcal B^{++} \mathcal C^{++} \mathcal A^{++} \mathcal B^{-+} \mathcal C^{--} \mathcal A^{-+} \mathcal B^{+-} \mathcal C^{+-} \mathcal A^{+-} \mathcal B^{--} \mathcal C^{-+} \ldots
\]
depending on $f$ is positive or negative.
\item[(4)] if $r$ and $t$ both are odd, then the generalized renormalization sequence $\{g_n\}$ exhibits as
\[
\mathcal A^{++} \mathcal A^{++} \mathcal A^{++} \mathcal A^{++} \ldots
\]
or
\[
\mathcal C^{-+} \mathcal A^{++} \mathcal A^{++} \mathcal A^{++} \ldots
\]
depending on $f$ is positive or negative.
\end{itemize}
The Fibonacci bimodal map studied in \cite{V} corresponds to stationary combinatorics $(2,1)$. In this case $\hat S_n = S_{n-1}$ and hence $S_{n+1} = S_n + S_{n-1}$. Therefore the first return time of $c$ and $d$ to $I^{n-1} \cup J^{n-1}$ coincide with the Fibonacci sequence $2, 3, 5, \ldots $.

\section{Yoccoz puzzle}

Let $f = a_d z^d + a_{d-1}z^{d-1}+ \ldots + z_0$ be a polynomial of degree $d \geq 2 $ with $a_d \neq 0$. Then $f$ has a superattracting fixed point at $\infty$. The basin of attraction of $\infty$, which is defined as 
\[ A(\infty): = \{ z \in \mathbb C| f^n(z) \to \infty \mbox{ as } n \to \infty \}
\]
is connected by the maximum principle. The filled Julia set $K(f)$ of $f$ is defined as the complement of $A(\infty)$. The Julia set $J(f)$ of $f$ is defined as $\partial K(f)$.

The classic B$\ddot{\rm o}$ttcher Theorem provides us a local holomorphic change of coordinate $\varphi(z)$ which conformally conjugate $f$ near $\infty$ to $z \to z^d$. The map $\varphi$ is tangent to the identity at $\infty$ and is unique up to multiplication. If the filled Julia set $K(f)$ contains all of the finite critical points of $f$, then both $K(f)$ and $J(f)$ are connected. The B$\ddot{\rm o}$ttcher function $\varphi(z)$ extends analytically to the whole basin of $\infty$, $\mathbb C \setminus K(f)$, and maps it conformally onto $\mathbb C \setminus \overline {\mathbb D}$ where $\mathbb D$ is the unit disk.

The Green function $G := \log |\varphi|$ extends harmonically to $\mathbb C \setminus K(f)$ and vanishes on $K(f)$. The level curve $\{ G(z) = r\}$, $r > 0$ is called the {\it equipotential curve of level r}, and denoted by $E(r)$. The {\it external ray of angle} $t \in \mathbb R / \mathbb Z$ is the gradient curve of $G$ stemming from infinity with the angle $t$ (measured via the B$\ddot{\rm o}$ttcher coordinate $\varphi$), and denoted by $R(t)$. Note the equation $G(f(z)) = G(z)d$ shows that $f$ maps each equipotential curve $E(r)$ to the equipotential curve $E(rd)$ by an $d$-to-1 covering map. When $J(f)$ is connected, $R(t) = \varphi^{-1}(\{ r e^{2\pi i t}| r >1\}) $. In this case, any external ray $R(t)$ with $t$ rational has a well-defined landing point $\lim_{r \to 1} \varphi^{-1}(r e^{2 \pi i t})$ which is contained in $J(f)$; vice versa, a repelling or parabolic point is the common landing point of finitely many external rays with rational angle.

A topological disk $D \subset \mathbb C$ is called nice if $int D \cap f^k(\partial D) = \emptyset, k=0,1,2,\ldots$. A set $X \subset \mathbb{C}$ is called $\mathbb{R}$-{symmetric} if it is invariant under complex conjugacy $z \to \overline{z}$. A holomorphic map $h : U \to \mathbb C$ is called $\mathbb{R}$-{symmetric} if $U$ is $\mathbb{R}$-symmetric and $h(\overline z) = \overline h(z)$. Note that $h$ preserves the real line. When $f$ is a real polynomial, the corresponding B$\ddot{\rm o}$ttcher coordinate $\varphi(z)$ and the Green function $G(z)$ are $\mathbb{R}$-{symmetric}.

Let $f \in \mathscr C$. Since the orbits of $c$ and $d$ are bounded, $J(f)$ is connected (and also locally connected).  Since $f$ is a real polynomial, if $R(t)$ lands at $z \in J(f)$, then its mirror image $R(-t)$ lands at $\overline z \in J(f)$. By Lemma 2.6, there exists a fixed point $p$ between $c$ and $d$ in the real line. By Singer's theorem \cite{MS} $p$ is hyperbolic repelling. In particular, $p$ is contained in the interior of $K(f) \cap \mathbb R$. Then by Lemma 5.2 in \cite{KSS}, there are exactly two external rays landing at $p$.

Now let us play the Yoccoz puzzle game for $f$. The game starts by cutting the complex plane with external rays. Suppose that $\mathcal R_1 = R(\theta)$ and $\mathcal R_2 = R(-\theta)$ landing at $p$. Select some equipotential curve surrounding the critical values $\{f(c), f(d)\}$, for example the curve $E(1/3)$.  Let $X^{(0)}$ be the open topological disk bounded by $E(1/3)$. The Yoccoz puzzle of $f$ is defined as the following sequence of graphs:
\[Y^{(0)} = \partial X^{(0)} \cup (X^{(0)} \cap \bigcup_{i=1,2} \overline{\mathcal R_i}),
\]
\[Y^{(n)} = f^{-n} Y^{(0)}, \  n =1,2,\ldots.
\]
A component of $f^{-n}X^{(0)} \setminus Y^{(n)}$ will be called {\it a puzzle piece of depth n}. A puzzle piece of depth $n$ which contains a point $z$ will be denoted by $P^{(n)}(z)$. For $n \geq 1$, the $f$-image of a puzzle piece of depth $n$ is a puzzle piece of depth $n-1$ and each puzzle piece of depth $n$ is contained in a puzzle piece of depth $n-1$.

Since $J(f)$ is connected and the orbits of $c$ and $d$ do not land at $p$, all puzzles $Y^{(n)}$ are well-defined. Since bounded by equipotentials and external rays, any puzzle piece $P$ is nice. Moreover, any two puzzle pieces $P$ and $Q$ are either nested or have disjoint interiors.

Consider a puzzle piece $P$ of depth $n \geq 1$. One can also define the first entry map and first return map to $P$. Let $Q$ be entry domain of $P$ and $f^m : Q \to P$ be a branch of the first entry map, then $Q$ is a puzzle piece of depth $n+m$.

We now introduce the {\it twin principal nest} of cubic polynomial $f \in \mathscr C$. Let $U^0$ and $V^0$ be puzzle pieces of depth $1$ so that $c \in U^0$ and $d \in V^0$.

\begin{lem}
$U^0$ and $V^0$ are $\mathbb R$-symmetric topological disks such that $U^0 \cap \mathbb R = I^0$ and $V^0 \cap \mathbb R = J^0$. 
\end{lem}

\begin{proof}
It suffices to prove for $U^0$. Since $\mathcal R_1$ and $\mathcal R_2$ landing at $p$, there exist $\mathcal R_1'=R(\theta')$ and $\mathcal R_2'=R(-\theta')$ landing at $p_1$, where $f(p_1) = p$ and $I^0=(p_1, p)$. They are clearly symmetric with respect to the real line. In fact, $U^0$ is the region bounded by these four external rays and the curve $f^{-1} E(1/3)$, hence is $\mathbb R$-symmetric and $U^0 \cap \mathbb R = I^0$.
\end{proof}

Since $c$ and $d$ are recurrent, define 
\[ 
U^0 \supset U^1 \supset U^2 \supset \ldots \supset \{ c\} \mbox{ and } V^0 \supset V^1 \supset V^2 \supset \ldots \supset \{ d\}
\]
inductively such that, for $n \geq 1$, $U^n$ and $V^n$ are first return domains to $U^{n-1} \cup V^{n-1}$. Then $\{ U^n \}_{ n \geq 0}$ is a sequence of nested puzzle pieces with depth $1 < n_1 < n_2 < \ldots$. Furthermore, $U^i$ is $\mathbb R$-symmetric and $U^i \cap \mathbb R = I^i$ for all $i \geq 0$. The same holds for $\{ V^n \}_{ n \geq 0}$. The two sequences of puzzle pieces will be referred to as the (complex) {\it twin principal nest} of $f$.

Let $\phi_n$ denote the first return map to $U^{n-1} \cup V^{n-1}$, then $\phi_n$ is the analytic extension of $g_n$. Therefore the $n$-th generalized renormalization $g_n : C^n \cup I^n \cup J^n \cup D^n \to I^{n-1} \cup J^{n-1}$ can be extended analytically to a complex box mapping $\phi_n : L^n \cup U^n \cup V^n \cup R^n \to U^{n-1} \cup V^{n-1}$. Where $L^n$ is a $\mathbb R$-symmetric puzzle piece with $L^n \cap \mathbb R = C^n$ and $R^n$ is the $\mathbb R$-symmetric puzzle piece with $R^n \cap \mathbb R = D^n$. The combinatorial type and inducing step are defined the same as real box mappings. Clearly, $\phi_n$ and $g_n$ have the same combinatorial type.

\section{Decay of geometry}

Suppose $f \in \mathscr C$ has combinatorial sequence $\mathcal S = \{( \theta_n, r_n, t _n)\}_{n \geq 1}$. By Yoccoz puzzle, let 
\[ 
U^0 \supset U^1 \supset U^2 \supset \ldots \supset \{ c\} \mbox{ and } V^0 \supset V^1 \supset V^2 \supset \ldots \supset \{ d\}
\]
be its (complex) twin principal nest. For $n \geq 1$, take $A^n : = U^{n-1} \setminus \overline{U^n}$ and $B^n : = V^{n-1} \setminus \overline{V^n}$. The annuli $A^n$ and $B^n$ are known as {\it fundamental annuli}. (We will refer in the same way to the corresponding open and semi-open annuli as well.) The {\it principal moduli} is defined as: $$\mu_n = \min\{ \m A^n, \m B^n \}.$$

The aim of this section is to prove the following proposition from which Theorem 1 will be deduced.

\begin{prop}
Suppose $f \in \mathscr C$ with $\mu_1 \geq \tau$. Then there exists a constant $C = C(\tau, f) >0$ such that $\mu_n \geq C \cdot n$ for every $n \geq 1$.
\end{prop}

\begin{proof}[Proof of Theorem 1]

It follows from Teichm$\ddot{\textrm u}$ller's module Theorem from \cite{LV}[pp. 56] and the above proposition that there exists $C_1 >0$ depends on $C$ such that 
\[
\log \frac{1}{\lambda_n} \geq C_1 \cdot n.
\]
This finishes the proof.

\end{proof}

\subsection{Moduli and Conformal Roughness} A topological annulus is a doubly connected domain $A \subset \mathbb{C}$. For $1 < r \leq \infty$, the {\it conformal modulus}, or simply {\it modulus}, of the round annulus $A_r = \{z \in \mathbb{C} : 1 < |z| <r \}$ is defined to be $\m A_r = \log{r}/(2\pi)$. Now, given any topological annulus $A$ which is not equal to a punctured disk or plane, there exists a conformal equivalence between $A$ and some round annulus $A_r$; hence we define $\m A = \m A_r$. Let $A_1$ and $A_2$ be annuli and $f: A_1 \to A_2$ be a holomorphic covering of degree $d$. Then $\m A_2  = d \cdot \m A_1 $. We mark here that if $A$ is a semi-open annulus, then $\m A$ is the same as $\m A^{\circ}$ where $A^{\circ}$ is the interior of $A$.

Two annuli $A_1$ and $A_2$ are called {\it nested} provided that they are disjoint and one separates the other from $\infty$. Let $A_1 \oplus A_2$ represent the smallest annulus containing both $A_1$ and $A_2$. The {\it super-additivity} of the modulus means that:
\[ \m(A_1 \oplus A_2) \geq \m A_1 + \m A_2.
\]

The following lemma concerns the situation when a mapping from one annulus onto another is holomorphic and proper but not a covering. The proof can be found in \cite{GS}.

\begin{lem}
Consider a topological disk $D_1'$, an annulus $U_1 \subset D_1'$ and the topological disk $D_1$ determined as the union of $U_1$ and the bounded component of the complement of $\overline U_1$. Denote $W_1 : =D_1' \setminus D_1 $. Suppose $f : D_1' \to D_2'$ is a degree 2 branched covering in the form $h \circ z^2$ where $h$ is univalent. Assume that $f$ is univalent on $D_1$. Define $D_2 : = f(D_1)$, $W_2 : = D_2' \setminus \overline D_2$ and $U_2 : = f(U_1)$. Choose non-negative numbers $\sigma_1$ and $\sigma_2$ such that
\begin{align*}
\sigma_2 & \leq \textrm{mod} \ U_2 \\
\sigma_1 & \leq \m \ U_2 + \m \ W_2.
\end{align*}
Then,
\begin{itemize}
\item[(1)] $\m \ U_1 + \m \ W_1 \geq \frac{1}{2}(\sigma_1+ \sigma_2)$;
\item[(2)] for every $\delta >0$ there is an $\epsilon >0$ such that if $\sigma_1 -\sigma_2 \geq \delta$ and $\m \ U_2 \geq \delta$, then
\[ \m (U_1 \oplus W_1) \geq \frac{\sigma_1 + \sigma_2}{2} + \epsilon.
\]
\end{itemize}
\end{lem}

It turns out that the inequality $\m(A_1 \oplus A_2) \geq \m A_1 + \m A_2$ becomes sharp based on some properties of the curve separating $A_1$ from $A_2$. To study this phenomenon, Graczyk and \'Swi\c atiek introduced the concept of `conformal roughness'.

\begin{definition}
Let $w \subset \mathbb C$ be a Jordan curve. We say that $w$ is $(M,\epsilon)$-rough, $M,\epsilon \geq 0$, if for every pair of open annuli $A_1$ and $A_2$ satisfying the following conditions
\begin{itemize}
\item $A_1$ is contained in the bounded component of the complement of $w$,
\item $w$ is contained in the bounded component of the complement of $A_2$,
\item $\m A_1, \m A_2 \geq M$,
\end{itemize}
the inequality
\[\m(A_1 \oplus A_2) > \m\ A_1 + \m \ A_2 + \epsilon
\]
holds.
\end{definition}

The following facts can be found in \cite[Section 4.2]{GS}.

\begin{fact}
Let $U \subset W$ be two annuli that share the same inner boundary. Let $w$ be the outer boundary of $W$. Assume that $\m W \geq \Delta$ and the distance between the two connected components of $W^c$ is at least $\lambda \cdot \diam \ w$. For every $\Delta, \lambda, M, \rho>0$ there are $\epsilon, \epsilon'>0$ so that if $\mbox{\m U} + \epsilon' > \mbox{\m W}$, then either $w$ is $(M,\epsilon)$-rough or the Hausdorff distance between $U$ and $w$ is less than $\rho \cdot \diam\ w$.
\end{fact}

\begin{fact}
Let $w \subset \mathbb C$ be Jordan curve which is invariant under the rotation $2\pi/\ell$ about $0$, $\ell \in \mathbb N$. For every $M >0$, there are $L, \epsilon >0$ so that if $|z_1|/|z_2| \geq L$ for some $z_1, z_2 \in w$, then $w$ is $(M,\epsilon)$-rough.
\end{fact}

\begin{fact}
Let $w \subset \mathbb C$ be Jordan curve, separating $0$ from $\infty$ and invariant under the rotation $2\pi/\ell$ about $0$, $\ell \in \mathbb N$. Let $U_{\delta} : = \{ z \in \mathbb C: \dist(z,w) \leq \delta \cdot \diam\ w \}$. Assume that function $\psi$ is defined on $U_{\delta}$ in the form $h \circ z^{\ell}$ with $h$ univalent. For every $\ell$ and $M, \delta, \epsilon_1 >0$, there are $M_1,\epsilon>0$, where $M_1$ depending only on $M$ and $\delta$, so that if $\psi(w)$ is $(M_1, \epsilon_1)$-rough, then $w$ is $(M, \epsilon)$-rough.
\end{fact}

\subsection{Separating Index} Let $\phi$ be a complex box mapping, $\phi : U_0 \cup U_1 \cup V_0 \cup V_1 \to U \cup V$. If $\phi$ has a combinatorial type, then $\phi$ has two post-critical branches $U_1$ and $V_1$. In case that $\phi(U_1) = U$, denote $U_P = \phi^{-1}(U_0)$; and in case that $\phi(U_1) = V$, denote $U_P = \phi^{-1}(V_0)$. Analogously we obtain $V_P \subseteq V_1$.

\noindent
{\bf Separating annuli.} The {\it separating annuli} of $\phi$ are any 10 annuli $A_i$ and $B_i$, $1 \leq i \leq 5$, either open or degenerated to Jordan curves, that satisfying the following conditions:
\begin{itemize}
\item $A_i \subset U$ and $B_i \subset V$ for $1 \leq i \leq 5$;
\item $A_2$ is surrounding $U_0$ and disjoint with $U_1$;
\item $A_1$ is taken as the intersection of $U$ and the unbounded component of the complement of $\overline A_2$;
\item $A_3$ is surrounding $U_1$ and disjoint with $U_0$;
\item $A_4$ is taken as the intersection of $U$ and the unbounded component of the complement of $\overline A_3$
\item $A_5$ is uniquely determined as $U_1 \setminus \overline U_P$;
\item $B_i$ are taken analogously.
\end{itemize}

\noindent
{\bf Separation symbols.} Let $\omega_i : = \min\{\m A_i, \m B_i\}$, $1 \leq i \leq 5$. A separation symbol $\Sigma$ is a choice of separating annuli as described above together with a quadruple of real numbers $(s_1,s_2,s_3,s_4)$ so that the following inequalities hold:
\begin{align*}
s_1 &\leq \omega_1 + \omega_2\\
s_2 &\leq \omega_2\\
s_3 &\leq \omega_3 + \omega_5\\
s_4 &\leq \omega_3 + \omega_4 + \omega_5.
\end{align*}

\noindent
{\bf Normalized symbol.} Choose $\beta>0$, set $\alpha:=\beta/2$. Pick $\lambda_1$ and $\lambda_2$ so that:
$$-\alpha \leq \lambda_1, \lambda_2 \leq \alpha \mbox{ and } \lambda_1 + \lambda_2 \geq 0.$$
If a seperation symbol $\Sigma$ satisfies:
\begin{align*}
s_1 & = \alpha + \lambda_1\\
s_2 & = \alpha - \lambda_2\\
s_3 & = \beta - \lambda_1\\
s_4 & = \beta + \lambda_2.
\end{align*}
Then we will say that $\Sigma$ is normalized with norm $\beta$ and corrections $\lambda_1$ and $\lambda_2$.

\subsection{Complex Bounds} In this subsection, we will begin computing the separation index for complex box mapping after one step inducing. Suppose $\varphi_1$ is a complex box mapping which has combinatorial type $(r, t)$. Let $\varphi_2$ be the complex box mapping induced from $\varphi_1$. Set up so that $\varphi_i : U_0^i \cup U_1^i \cup V_0^i \cup V_1^i \to U_0^{i-1} \cup V_0^{i-1}$, $i = 1, 2$. Keep in mind that this is the case when $\varphi_1$ is the $n$-th generalized renormalization for some $n$.

\begin{lem}
There exist topological disks $W_c \supseteq U_0^2$ and $W_d \supseteq V_0^2$ such that $\varphi_2 : W_c \cup W_d \to U_0^0 \cup V_0^0$ is degree 2 branched covering on each component.
\end{lem}

\begin{proof}

We may assume that $\varphi_2$ is of type $\mathcal B$, namely $\varphi_2 : U_0^2 \to U_0^1$ is degree 2 branched covering. Note that $\varphi_2|U_0^2 = \varphi_1^r|U_0^2$. We pullback the topological disk $U_0^0$ along $\{ \varphi_1^i(c)\}_{i =0}^r$ and denoted by $W_c$ the disk produced. The pullback is certainly unimodal, hence $\varphi_2 : W_c \to U_0^0$ is degree 2 branched covering. The disk $W_c$ produced is contained in $U_0^1$ and contains $U_0^2$. The construction of $W_d$ is similar. This finishes the proof.

\end{proof}

\begin{lem}
Consider a normalized separation symbol
\[ \Sigma = (s_1=\alpha+\lambda_1, s_2 =\alpha-\lambda_2, s_3 = \beta-\lambda_1,s_4=\beta+\lambda_2 )
\]
and suppose that $(s_1,s_2,s_3+\epsilon,s_4+\epsilon)$ is another normalized separation symbol with $\epsilon>0$. Then there is a normalized separation symbol with norm $\beta + {\epsilon}/{2}$.
\end{lem}

\begin{proof}
See \cite[Lemma 4.1.2]{GS}.
\end{proof}

\begin{lem}
Suppose that $\varphi_1$ has combinatorial type $(r, t)$ with $t=1$. Assume that $\varphi_1$ has a normalized separation symbol $\Sigma^{(1)} = (s_1^{(1)},s_2^{(1)},s_3^{(1)},s_4^{(1)})$ with norm $\beta$ and corrections $\lambda_1$ and $\lambda_2$. Then $\varphi_2$ has a normalized separation symbol $\Sigma^{(2)}$ with norm $\beta$ and corrections 
$$\lambda_1'=\frac{\lambda_2}{2} \mbox{ and } \lambda_2' = \frac{\lambda_1}{2}.$$
In particular, we have $s_2^{(2)} \geq \beta/4$.
\end{lem}

\begin{proof}

Let $A_i^{(1)}$ and $B_i^{(1)}$, $1 \leq i \leq 5$, be the separating annuli for $\varphi_1$.  We shall construct the separating annuli $A_i^{(2)}$ and $B_i^{(2)}$ for $\varphi_2$ by pullback. Without loss of generality, we may assume that $\varphi_1$ is of type $\mathcal C$. By Lemma 2.4 and Lemma 2.5,  $\varphi_2$ is always of type $\mathcal A$. Set $\psi: = \varphi_1|U_0^1$.

Since $\varphi_1$ is of type $\mathcal C$, $\varphi_1 : U_0^1 \to V_0^0$ and $\varphi_1 : V_0^1 \to U_0^0$ are degree 2 branched covering, while $\varphi_1 : U_1^1 \to U_0^0$ and $\varphi_1 : V_1^1 \to V_0^0$ are univalent. Since $t=1$, $\varphi_2|U_1^2 \cup V_1^2 = \varphi_1$, hence $U_1^2$ and $V_1^2$ are two immediate branches for $\varphi_2$. Since $\varphi_2$ is of type $\mathcal A$, $\varphi_1(U_1^2) = V_0^1$ and $\varphi_1(V_1^2) = U_0^1$. Looking into the proof of  Lemma 2.1 and 2.2, $\psi(U_0^2) \subseteq V_P^1 \subseteq V_1^1$. In particular, $\psi(U_0^2) = V_P^1$ holds when $r=2$. Hence $\m (V_1^1 \setminus \psi(U_0^2) ) \geq \m (V_1^1 \setminus V_P^1 ) = \m B_5^{(1)}$.

The separating annuli $A_i$, $1 \leq i \leq 4$, are taken as:
\begin{align*}
A_2^{(2)} & = \psi^{-1}(B_3^{(1)} \oplus (V_1^1 \setminus \psi(U_0^2)))\\
A_2^{(1)} & = \psi^{-1}(B_4^{(1)})\\
A_3^{(2)} & = \psi^{-1}(B_2^{(1)})\\
A_4^{(2)} & = \psi^{-1}(B_1^{(1)}).
\end{align*}
Then $\psi:A_2^{(2)} \to  B_3^{(1)}\oplus (V_1^1 \setminus \psi(U_0^2)) $ and $\psi : A_2^{(1)} \to B_4^{(1)}$ are degree 2 branched covering. Clearly $A_2^{(2)}$ is surrounding $U_0^2$. Since $B_3^{(1)}$ is disjoint with $V_0^1$, $A_2^{(2)}$ is disjoint with $U_1^2$.

Therefore,
\begin{align*} 
\m A_2^{(2)} & \geq \frac{1}{2}(\m B_3^{(1)} + \m B_5^{(1)}) \\
\m A_2^{(1)} & \geq \frac{1}{2}\m B_4^{(1)} \\
\m A_3^{(2)} & = \m B_2^{(1)} \\
\m A_4^{(2)} & \geq  \frac{1}{2}\m B_1^{(1)}.
\end{align*}
The annulus $A_5^{(2)}$ is conformally equivalent to $V_0^1 \setminus V_0^2$, hence its modulus is at least $s_1^{(2)}$. Construct $B_i^{(2)}$ in the same way, $1 \leq i \leq 5$. Then,

\begin{align*}
s_2^{(2)} &\geq \frac{1}{2}(\omega_3^{(1)} + \omega_5^{(1)})
\geq \frac{1}{2} s_3^{(1)} = \frac{1}{2}(\beta-\lambda_1) = \alpha-\frac{\lambda_1}{2}\\
s_1^{(2)} &= \omega_1^{(2)} +   \omega_2^{(2)}   \geq \frac{1}{2}s_3^{(1)} + \frac{1}{2} \omega_4^{(1)} = \frac{1}{2} s_4^{(1)} = \alpha + \frac{\lambda_2}{2}\\
s_3^{(2)} & = \omega_3^{(2)} + \omega_5^{(2)} \geq s_2^{(1)} + s_1^{(2)}
 = \alpha - \lambda_2 + \alpha +  \frac{\lambda_2}{2} = \beta -  \frac{\lambda_2}{2}\\
s_4^{(2)} &=\omega_3^{(2)} +\omega_4^{(2)} + \omega_5^{(2)} \geq s_1^{(1)} + s_1^{(2)} = \beta + \frac{\lambda_1}{2}.
\end{align*}
Since $-\alpha \leq \lambda_1 \leq \alpha$, then $s_2^{(2)} \geq \beta/2 -\lambda_2' \geq \beta/4$. This finishes the proof.

\end{proof}

\begin{lem}
Suppose that $\varphi_1$ has combinatorial type $(r, t)$ with $t \geq 2$. Assume that $\varphi_1$ has a normalized separation symbol $\Sigma^{(1)} = (s_1^{(1)},s_2^{(1)},s_3^{(1)},s_4^{(1)})$ with norm $\beta$ and corrections $\lambda_1$ and $\lambda_2$. If $t > r$, then $\varphi_2$ has a normalized separation symbol $\Sigma^{(2)}$ with norm $\beta$. Moreover, if $U_1^1$ and $V_1^1$, respectively, are separated from the boundary of $U_0^0$ and $V_0^0$ by annuli with modulus at least $L$, then $\varphi_2$ has a normalized separation symbol $\Sigma^{(2)}$ with norm $\beta+ {L}/{4}$.
\end{lem}

\begin{proof}

We may assume that $\varphi_1$ is of type $\mathcal C$. Then $\varphi_1 : U_0^1 \to V_0^0$ and $\varphi_1 : V_0^1 \to U_0^0$ are degree 2 branched covering, while $\varphi_1 : U_1^1 \to U_0^0$ and $\varphi_1 : V_1^1 \to V_0^0$ are univalent. Let $A_c$ and $A_d$ be the annuli separating $U_1^1$ and $V_1^1$ from the boundary of $U_0^0$ and $V_0^0$. Set $\psi = \varphi_1 | U_0^1$. Set $h = (\varphi_1|V_1^1)^{r-1}$ and $\tilde h = (\varphi_1|V_1^1)^{t-1}$.

Since $r, t \geq 2$, $\varphi_1( U_0^2 \cup U_1^2) \subseteq V_P^1 \subseteq V_1^1$. We first pullback $V_0^1$ and $V_1^1$ under $\varphi_1|V_1^1$. Let $\tilde W$ be the topological disk contained in $V_1^1$ such that $h : \tilde W \to V_0^0$ is univalent. Then $\tilde W \subset V_P^1$. Let $W \subset \tilde W$ be the preimage of $V_0^1$ under $h$, let $Z \subset \tilde W$ be the preimage of $V_1^1$ under $h$. Then $W \cap Z = \emptyset$. Moreover, $\psi(U_0^2) = W$ and $\psi(U_1^2) \subsetneqq Z$.

The separating annuli $A_i$, $1 \leq i \leq 4$, are taken as:
\begin{align*}
A_2^{(2)} & = \psi^{-1} \circ h^{-1} (B_2^{(1)})\\
A_1^{(2)} &=  \psi^{-1} \circ h^{-1} (B_1^{(1)}) \oplus \psi^{-1} (A_d)\\
A_3^{(2)} & = \psi^{-1}( h^{-1} (B_3^{(1)}) \oplus (Z \setminus \psi(U_1^2))) \\
A_4^{(2)} & = \psi^{-1} \circ  h^{-1} (B_4^{(1)}).
\end{align*}
Since $B_2^{(1)}$ is surrounding $V_0^1$ with $B_2^{(1)} \cap V_1^1 = \emptyset$, then $h^{-1}(B_2^{(1)})$ is surrounding $W$ and is disjoint with $Z$. Since $h : \tilde W \to V_0^0$ is univalent and onto, $h^{-1}(B_1^{(1)})$ is the intersection of $\tilde W$ and the unbounded component of the complement of the closure of $h^{-1}(B_2^{(1)})$. Also, $h^{-1}(B_1^{(1)})$ is contained in the bounded component of the complement of $\overline A_d$ (if exists). Similarly, $h^{-1}(B_3^{(1)})$ is surrounding $Z$ with $h^{-1}(B_3^{(1)}) \cap W = \emptyset$. Then after pullback, $A_2^{(2)}$ is surrounding $U_0^2$ which is disjoint with $U_1^2$ and such that $h \circ \psi: A_2^{(2)} \to B_2^{(1)}$ is degree 2 branched covering. Since the annulus $h^{-1}(B_3^{(1)})$ is not surrounding the critical value $\psi(c)$, $\psi:A_3^{(2)} \to \psi^{-1}( h^{-1} (B_3^{(1)}) \oplus (Z \setminus \psi(U_1^2))$ is univalent. On the other hand, $h \circ \psi :A_4^{(2)} \to B_4^{(1)}$ is holomorphic while $A_4^{(2)}$ contains a critical point. Finally, by pullback the annulus $V_0^0 \setminus V_0^1$ under univalent branch $\tilde h$, we obtain an annulus contained in $Z$ and surrounds $\psi(U_1^2)$ with modulus $\m(V_0^0 \setminus V_0^1) = \m B_5^{(1)}$. Then $\m (Z \setminus \psi(U_1^2)) \geq \m B_5^{(1)}$.

Therefore,
\begin{align*}
\m A_2^{(2)} & = \frac{1}{2} \m B_2^{(1)} \geq \frac{1}{2} s_2^{(1)}  \\
\m A_1^{(2)} & \geq \frac{1}{2} (\m B_1^{(1)} + L ) \geq \frac{1}{2} (\omega_1^{(1)} + L )\\
\m A_3^{(2)} & \geq \m B_3^{(1)} + \m B_5^{(1)} \geq s_3^{(1)}\\
\m A_4^{(2)} & \geq \frac{1}{2} \m B_4^{(1)} \geq \frac{1}{2} \omega_4^{(1)}.
\end{align*}
The annulus $A_5^{(2)}$ is conformally equivalent to $V_0^1 \setminus V_0^2$, hence its modulus is at least $s_1^{(2)}$. Construct $B_i^{(2)}$ in the same way, $1 \leq i \leq 5$. The separation symbol goes:
\begin{align*}
s_1^{(2)} & = \frac{\alpha+\lambda_1}{2} + \frac{L}{2}    \\
s_2^{(2)} & = \frac{\alpha-\lambda_2}{2}\\
s_3^{(2)} & = \beta + \frac{\alpha-\lambda_1+L}{2}\\
s_4^{(2)} & = \beta + \frac{\alpha+\lambda_2+L}{2}.
\end{align*}

Now let $\lambda_1'=\frac{\lambda_1-\alpha}{2}$, $\lambda_2'=\frac{\lambda_2+\alpha}{2}$. These estimates yield a normalized separation symbol with norm $\beta$ when $L=0$. By Lemma 4.3, the presence of these extra terms allows to lift the norm $\beta$ to $\beta + \frac{L}{4}$.

\end{proof}

\begin{lem}
Suppose that $\varphi_1$ has combinatorial type $(r, t)$ with $t \geq 2$. Assume that $\varphi_1$ has a normalized separation symbol $\Sigma^{(1)} = (s_1^{(1)},s_2^{(1)},s_3^{(1)},s_4^{(1)})$ with norm $\beta$ and corrections $\lambda_1$ and $\lambda_2$. If $t < r$, then $\varphi_2$ has a normalized separation symbol $\Sigma^{(2)}$ with norm $\beta$. Moreover, if $U_1^1$ and $V_1^1$, respectively, are separated from the boundary of $U_0^0$ and $V_0^0$ by annuli with modulus at least $L$, then $\varphi_2$ has a normalized separation symbol $\Sigma^{(2)}$ with norm $\beta+ {L}/{4}$.
\end{lem}

\begin{proof}

This case is similar  to the case covered by Lemma 4.4. For simplicity, we still assume that $\varphi_1$ is of type $\mathcal C$. Let $A_c$ and $A_d$ be the annuli separating $U_1^1$ and $V_1^1$ from the boundary of $U_0^0$ and $V_0^0$. Set $\psi = \varphi_1 | U_0^1$. Set $h = (\varphi_1|V_1^1)^{r-1}$ and $\tilde h = (\varphi_1|V_1^1)^{t-1}$.

Let $\tilde W$ be the topological disk contained in $V_1^1$ such that $\tilde h : \tilde W \to V_0^0$ is univalent. Then $\tilde W \subset V_P^1$. Let $W \subset \tilde W$ be the preimage of $V_0^1$ under $\tilde h$, let $Z \subset \tilde W$ be the preimage of $V_1^1$ under $\tilde h$. Then $W \cap Z = \emptyset$. Moreover, $\psi(U_1^2) = W$ and $\psi(U_0^2) \subsetneqq Z$.

The separating annuli $A_i$, $1 \leq i \leq 4$, are taken as:
\begin{align*}
A_2^{(2)} & = \psi^{-1} ( \tilde {h}^{-1} (B_3^{(1)}) \oplus (Z \setminus \psi(U_0^2)) )\\
A_1^{(2)} & = \psi^{-1} \circ \tilde {h}^{-1} (B_4^{(1)}) \oplus \psi^{-1} (A_d)\\
A_3^{(2)} & = \psi^{-1} \circ \tilde {h}^{-1} (B_2^{(1)} )\\
A_4^{(2)} & = \psi^{-1} \circ \tilde {h}^{-1} (B_1^{(1)}).
\end{align*}
Then $\psi : A_2^{(2)} \to  \tilde {h}^{-1} (B_3^{(1)}) \oplus (Z \setminus \psi(U_0^2))$ is degree 2 branched covering. Since $\tilde {h}^{-1} (B_3^{(1)})$ is disjoint with $W$, $A_2^{(2)}$ is disjoint with $U_1^2$. Also, we have $(Z \setminus \psi(U_0^2)) \geq \m(V_0^0 \setminus V_0^1) = \m B_5^{(1)}$. Since $\tilde {h}^{-1} (B_2^{(1)} )$ is not surrounding the critical value $\psi(c)$, $\tilde h \circ \psi: A_3^{(2)} \to B_2^{(1)}$ is univalent while $ A_4^{(2)}$ contains a critical point. Therefore,
\begin{align*} 
\m A_2^{(2)} & \geq \frac{1}{2}(\m B_3^{(1)} + \m B_5^{(1)}) \\
\m A_2^{(1)} & \geq \frac{1}{2}\m B_4^{(1)} + \frac{L}{2} \\
\m A_3^{(2)} & = \m B_2^{(1)} \\
\m A_4^{(2)} & \geq  \frac{1}{2}\m B_1^{(1)}.
\end{align*}
Construct $B_i^{(2)}$, $1 \leq i \leq 4$ in the same way. Then argue as in the proof of Lemma 4.4, this implies the existence of a normalized separation symbol with norm $\beta$ when $L=0$. The existence of $A_c $ and $A_d$ give an extra growth of $L/2$ in $s_1^{(2)}$, hence in $s_3^{(2)}$ and $s_4^{(2)}$. By Lemma 4.3, the norm can be lifted to $\beta+L/4$.
\end{proof}

\begin{prop}[Complex a priori bounds]
Suppose $f \in \mathscr C$ with $\mu_1 \geq \tau$. Then there exists a constant $\delta = \delta(\tau) >0$ such that $\mu_n \geq \delta$ for every $n \geq 1$.
\end{prop}

\begin{proof}
From the beginning, consider the map $\phi_1$. The domain of $\phi_1$ consists of four $\mathbb R-$symmetric topological disks. Take $A_2^{(1)},B_2^{(1)},A_3^{(1)},B_3^{(1)}$ degenerate, then we obtain a separation symbol $(\tau, 0, \tau, \tau)$. Let $\beta = \tau/2$, $\lambda_1^{(1)}=0$ and $\lambda_2^{(1)} = \tau/4$. This leads to a normalized separation symbol with norm $\tau/2$. By Lemma 4.4, Lemma 4.5 and Lemma 4.6, each generalized renormalization $\phi_n$ is a complex box mapping which has a normalized separation symbol with norm $\tau/2$ for all $n \geq 1$. Therefore $\mu_n \geq s_4^{(n-1)}/2 \geq \tau/8$ for $n \geq 2$. 
\end{proof}

\subsection{Linear growth of rate} To prove Proposition 3, it suffices to show that the norm $\beta$ has a growth with definite rate $\epsilon>0$ after several steps of inducing. Given any complex box mapping $\phi$, we say that $\phi$ has {\it separation bounds $\delta$} provided that
\[
\min \{ \m(U \setminus U_0), \m (V \setminus V_0)\} \geq \delta \mbox{ and } \min \{ \m(U \setminus U_1), \m (V \setminus V_1)\} \geq \frac{\delta}{2}.
\]

\begin{lem}
Suppose that $\varphi_1$ has a combinatorial type. Assume that $\min \{ \m(U_0^0 \setminus U_0^1), \m (V_0^0 \setminus V_0^1)\} \geq \delta$, then $\min \{ \m(U_0^1 \setminus U_1^2), \m (V_0^1 \setminus V_1^2)\} \geq \delta/2$.
\end{lem}

\begin{proof}
We may assume that $\varphi_1$ is of type $\mathcal C$. If $t=1$, then $\varphi_1 : U_0^1 \setminus U_1^2 \to V_0^0 \setminus V_0^1$ contains a critical point. Hence $\m (U_0^1 \setminus U_1^2) \geq \m (V_0^0 \setminus V_0^1) /2 \geq \delta/2$. If $t \geq 2$, looking into the proof of Lemma 4.5 and Lemma 4.6, then $\varphi_1(U_1^2) \subsetneqq V_1^1$ with $\m (V_1^1 \setminus \varphi_1(U_1^2) ) \geq \m B_5^{(1)} \geq \delta$. We can conclude that $\m(U_0^1 \setminus U_1^2) \geq \delta/2 $. Analogously we have $\m (V_0^1\setminus V_1^2) \geq \delta/2$. This finishes the proof.
\end{proof}

Proposition 4 and Lemma 4.7 together show that for each $n \geq 2$, the $n$-th level generalized renormalization always has a separation bounds for $\delta = \beta/4$.

\begin{cor}
Suppose that $\varphi_1$ has combinatorial type $(r, t)$ with $t \geq 2$. Assume that $\varphi_1$ has a normalized separation symbol with norm $\beta$ and separation bounds $\delta$. Then $\varphi_2$ has a normalized separation symbol with norm $\beta+ \delta/8$.
\end{cor}

\begin{proof}
This corollary follows from the second claim of Lemma 4.5 and Lemma 4.6 by taken $L = \delta /2$. 
\end{proof}

\begin{lem}
Suppose that $\varphi_1$ has combinatorial type $(r, t)$ with $r \geq 3$ and $t=1$. Assume that $\varphi_1$ has a normalized separation symbol with norm $\beta$ and separation bounds $\delta$. Then we can choose the separating annuli $A_2^{(2)}$ and $B_2^{(2)}$ for $\varphi_2$ so that
\[ \min\{ \m A_2^{(2)}, \m B_2^{(2)} \} \geq s_2^{(2)} + \frac{\delta}{4},
\]
where $s_2^{(2)}$ is given in Lemma 4.4.
\end{lem}

\begin{proof}
We may assume that $\varphi_1$ is of type $\mathcal C$. Looking into the proof of Lemma 4.4, $A_2^{(2)}$ is taken as $\psi^{-1}(B_3^{(1)} \oplus (V_1^1 \setminus \psi(U_0^2))$. Since $r \geq 3$, let $W_1 = (\varphi_1 |V_1^1)^{-1} (V_1^1)$ and $W_2 = (\varphi_1 |V_1^1)^{-(r-2)} (V_1^1)$. Then $\m (V_1^1 \setminus W_1) = \m (V_0^1 \setminus V_1^1) \geq \delta /2$ and $\m (W_2 \setminus \psi(U_0^2)) = \m B_5^{(1)}$. Therefore,

\begin{align*}
\m A_2^{(2)} & \geq \frac{1}{2} [  \m B_3^{(1)} + \m (V_1^1 \setminus W_1) +  \m (W_2 \setminus \psi(U_0^2))] \\
& \geq \frac{1}{2} [  \m B_3^{(1)} + \m B_5^{(1)} + \frac{\delta}{2} ]\\
& \geq  s_2^{(2)} + \frac{\delta}{4}.
\end{align*}
The same holds for $\m B_2^{(2)}$. This finishes the proof.
\end{proof}

\begin{lem}
Suppose that $\varphi_1$ has combinatorial type $(r, t)$ with $r \geq 3$ and $t=1$. Assume that $\varphi_1$ has a normalized separation symbol with norm $\beta$ and separation bounds $\delta$. Then $\varphi_2$ has a normalized separation symbol with norm $\beta+ \delta/8$.
\end{lem}

\begin{proof}
By Lemma 4.9 the separation symbol $s_2^{(2)}$ increases with rate $\delta/4$. Then $s_1^{(2)}$ can be increased by $\delta/4$, hence $s_3^{(2)}$ and $s_4^{(2)}$ can be increased by $\delta/4$. In this way we build a normalized separation symbol $\tilde \Sigma = (s_1^{(2)} + \delta/4, s_2^{(2)}+\delta/4, s_3^{(2)} + \delta/4,s_4^{(2)} + \delta/4)$. By Lemma 4.3, the norm can be lifted to $\beta + \delta/8$.
\end{proof}

If $\varphi_i$ has combinatorial type $(2,1)$, then we say that $\varphi_i$ is of Fibonacci type. The rest of this subsection is occupied for Fibonacci returns.

\begin{lem}
Suppose that $\varphi_1$ is induced from $\varphi_{-}$ such that both $\varphi_{-}$ and $\varphi_1$ have a normalized separation symbol with norm $\beta$ and separation bounds $\delta$. Suppose that $\varphi_1$ is of Fibonacci type and of type $\mathcal B$. Assume that $\varphi_1|U_0^1$ and $\varphi_1|V_0^1$ has the form $h_1 \circ (z-c)^2$ and $h_2 \circ (z-d)^2$. Then for every $M, \delta>0$, there exist $\eta_1,\eta_2>0$ such that the following alternative holds:
\begin{itemize}
\item $\m (U_0^0 \setminus U_P^1) \geq s_3^{(1)} + \eta_1 $;
\item $\partial U_0^1$ is $(M,\eta_2)$-rough.
\end{itemize}
The same holds for $\m (V_0^0 \setminus V_P^1)$ and $\partial V_0^1$.
\end{lem}

\begin{proof}
Since $\varphi_1$ is of type $\mathcal B$, $\varphi_1 : U_0^1 \to U_0^0$ is a degree 2 branched covering. Since $\varphi_1$ is of Fibonacci type,  $\varphi_1 : U_0^2 \to U_P^1$ is a degree 2 branched covering. Hence $\varphi_1 : U_0^1 \setminus U_0^2 \to U_0^0 \setminus U_P^1$ is also a degree 2 branched covering. Since $\varphi_1$ is induced from $\varphi_-$, by Lemma 4.2, $\varphi_1$ can be extended to $W_c \supseteq U_0^1$ with $\m(W_c \setminus U_0^1) \geq \delta/2$. Since $\partial W_c$ and $\partial U_0^1$ are both centrally symmetric with respect to $c$, there exists $\rho=\rho(\delta)$ so that $W_c $ contains the set $U_{\rho}: = \{ z \in \mathbb C : dist(z, \partial U_0^1) \leq \rho \cdot \diam \partial U_0^1 \}$.

Now suppose that $\m (U_0^0 \setminus U_P^1) \leq s_3^{(1)} + \eta'$, where $\eta' >0$ is a small parameter to be specified later. The aim is to show that if $\eta'$ is chosen appropriately depending on $M$ and $\delta$, then $\partial U_0^1$ is $(M, \eta'')$-rough.

Firstly, in the view of Fact 4.3, for every $M$, $\delta$ and $\epsilon_1>0$, there exist $M_0(M, \delta)$ and $\epsilon(M, \delta, \epsilon_1)$ positive such that if $\partial U_0^0$ is $(M_0, \epsilon_1)$-rough, then $\partial U_0^1$ is $(M, \epsilon)$-rough. Note that the parameter $\epsilon_1$ is chosen arbitrarily, while the other two are dependent as shown.

By Fact 4.2, for every $M_0>0$, there exist $L$ and $\epsilon_2>0$ so that if $|z_1-c|/|z_2-c| \geq L$ for some $z_1, z_2 \in \partial U_0^0$, then $\partial U_0^0$ is $(M_0, \epsilon_2)$-rough. We set $\epsilon_1:=\epsilon_2$. Then $\epsilon$ depends only on $M$ and $\delta$.

So let us assume that $\partial U_0^0$ is not $(M_0,\epsilon_1)$-rough, which implies for all $z_1,z_2\in \partial U_0^0$, we have 
\[ \frac{|z_1-c|}{|z_2-c|} \leq L.
\]
Let $\lambda_0 \cdot \diam U_0^0$ denote the Hausdorff distance between the connected components of the boundary of $A_4^{(1)}$. By Koebe's theorem, $\varphi_1|U_0^1$ can be composed as $h \circ (z-c)^2$ where $h$ has bounded distortion in terms of $\delta$. Then there exists $\lambda_1$ depending on $\delta$ and $\lambda_0$ such that the Hausdorff distance between the outer boundary of $A_2^{(2)}$ and $\partial U_0^1$ is less than $\lambda_1 \cdot \diam U_0^1$. Since the outer boundary of $A_4^{(1)}$ is centrally symmetric about $c$, hence for every $z \in A_4^{(1)}$ we have
\[ |z-c| \geq (\frac{1}{2L} - \frac{\lambda_0}{2}) \cdot \diam U_0^1 > 0
\]
provided $\lambda_0 < 1/L$. A contradiction since $A_4^{(1)}$ does not separate $c$ from $U_0^0$.

Now we go back to the beginning where we assume that $\m (U_0^0 \setminus U_0^1) \leq s_3^{(1)} + \eta'$. By super-additivity of moduli, we have
\[ \m (U_0^1 \setminus U_0^2) = \m( \varphi_1^{-1}(U_0^0 \setminus U_P^1)) \leq \frac{1}{2} s_3^{(1)} + \frac{1}{2} \eta' \leq \m A_2^{(2)}+ \frac{1}{2} \eta'.
\]
Apply Fact 4.1 in the following sense. Let $W = U_0^1 \setminus U_0^2$ and $U = A_2^{(2)}$. Then $\m W \geq \delta$ and the distance between the two connected components of $W^c$ is at least $\rho_1 \cdot \diam \partial W$ where $\rho_1$ depends only on $\delta$. Therefore, for every $M , \delta$ and $\rho_2>0$ (different with $\rho_1$) there exist $\epsilon_3$ and $\epsilon_4$ so that if 
\[ \m U + \epsilon_3 \geq \m W,
\]
then either
\begin{itemize}
\item $\partial U_0^1$ is $(M,\epsilon_4)$-rough or
\item the Hausdorff distance between the outer boundary of $A_2^{(2)}$ and $\partial U_0^1$ is less that $\rho_2 \cdot \diam U_0^1$.
\end{itemize}

We set $\lambda_0: = 1/(2L)$ and determine the corresponding $\lambda_1$ which depends only on $M$ and $\delta$. Set $\rho_2 := \lambda_1$. Set $\eta' := \epsilon_3$, hence $\eta'>0$ depends only on $M$ and $\delta$. We conclude that $\partial U_0^1$ is $(M,\eta_2)$-rough, where $\eta_2 : = \min \{\epsilon, \epsilon_4 \}$ depends only on $M$ and $\delta$, provided that $\m (U_0^0 \setminus U_P^1) \leq s_3^{(1)} + \eta'$. Finally, set $\eta_1: = \eta'$. This finishes the proof.

\end{proof}

\begin{lem}
Suppose that $\varphi_1$ is induced from $\varphi_{-}$ such that both $\varphi_{-}$ and $\varphi_1$ have a normalized separation symbol with norm $\beta$ and separation bounds $\delta$. Suppose that $\varphi_1$ is of Fibonacci type and of type $\mathcal B$. Assume that $\varphi_1|U_0^1$ and $\varphi_1|V_0^1$ has the form $h_1 \circ (z-c)^2$ and $h_2 \circ (z-d)^2$. Suppose that $s_4^{(1)} - s_3^{(1)} > \eta_1/2$ where $\eta_1$ is given in the previous lemma. Then there exists $\eta>0$ depending only on $\eta_1$ and $\delta$ such that $\varphi_4$ has a normalized separation symbol with norm $\beta+\eta$.
\end{lem}

\begin{proof}
Without loss of generality, we may assume that $\varphi_2$ and $\varphi_3$ are of Fibonacci type. By Lemma 2.4 and 2.5, $\varphi_2$ is of type $\mathcal C$ and $\varphi_3$ is of type $\mathcal A$. Looking into the proof of Lemma 4.4, the assumption $s_4^{(1)} - s_3^{(1)} > \eta_1/2$ implies $\varphi_2$ has a normalized separation symbol with norm $\beta$ and such that
\[ s_1^{(2)} -s_2^{(2)} = \lambda_1'+\lambda_2'=\frac{\lambda_1+\lambda_2}{2} > \frac{\eta_1}{4}.
\]

Since $\varphi_2$ is of type $\mathcal C$, $\varphi_2 : U_0^2 \to V_0^1$ is a degree 2 branched covering while $\varphi_2 : U_1^2 \to U_0^1$ is univalent. Since $\varphi_2$ is of Fibonacci type, $\varphi_3|U_0^3 = \varphi_2 \circ \varphi_2$ while $\varphi_3|U_1^3 = \varphi_2$. Therefore, $\varphi_2: U_0^3 \to V_P^2$ is a degree 2 branched covering while $\varphi_2 :U_1^3 \to V_0^2$ is univalent. See from the proof of Lemma 4.4, $A_3^{(3)}$ and $A_4^{(3)}$ are the preimages of $B_2^{(2)}$ and $B_1^{(2)}$ under $\varphi_2$. Then $\m A_3^{(3)} = \m B_2^{(2)}$ and $\m A_4^{(3)} \geq 1/2 \cdot \m B_1^{(2)}$.

We will apply Lemma 4.1 in the following context. Make $f$ equal to $\varphi_2|U_0^2$, $U_1$ equal to $A_3^{(3)} \oplus A_5^{(3)}$ and $D_1'=U_0^2$. Then make $U_2 = B_2^{(2)} \oplus (V_0^2 \setminus V_0^3)$ and $W_2 = B_1^{(2)}$. So we can apply to $\sigma_1 = s_1^{(2)} + s_1^{(3)}$ and $\sigma_2 = \m B_2^{(2)} + \m (V_0^2 \setminus V_0^3) = s_2^{(2)} + s_1^{(3)}$. Since $\m U_2 \geq \m (V_0^2 \setminus V_0^3) \geq \delta$ and $\sigma_1 - \sigma_2 \geq \eta_1/4$, there exists $\tilde \epsilon>0$ depending only on $\delta$ and $\eta_1$ such that
\[ \m( A_3^{(3)} \oplus A_4^{(3)} \oplus A_5^{(3)}  ) \geq \frac{s_1^{(2)} + s_2^{(2)}}{2} + s_1^{(3)} + \tilde \epsilon.
\]
Taking into account that $s_1^{(2)} = \alpha+\lambda_1'$, $s_2^{(2)} = \alpha-\lambda_2'$ and $s_1^{(3)} = \alpha + \lambda_2'/2$ as elements of a normalized separation symbol, we have
\[ \m( A_3^{(3)} \oplus A_4^{(3)} \oplus A_5^{(3)}  ) \geq \frac{\alpha+\lambda_1' + \alpha - \lambda_2'}{2} + \alpha + \frac{\lambda_2'}{2} + \epsilon' \geq s_4^{(3)} + \tilde \epsilon.
\]
Using the same argument we have 
\[ \min\{\m( A_3^{(3)} \oplus A_4^{(3)} \oplus A_5^{(3)}), \m( B_3^{(3)} \oplus B_4^{(3)} \oplus B_5^{(3)}  )  \} \geq s_4^{(3)} + \tilde \epsilon.
\]
The growth on $s_4^{(3)}$ increases $s_1^{(4)}$ with rate at least $\tilde \epsilon/2$, hence $s_3^{(4)}$ and $s_4^{(4)}$ can be increased by $\tilde \epsilon/2$. Therefore we build a normalized separation symbol $\tilde \Sigma = (s_1^{(4)} + \tilde \epsilon/2, s_2^{(4)}, s_3^{(4)} + \tilde \epsilon/2, s_4^{(4)} + \tilde \epsilon/2)$ for $\varphi_4$. By Lemma 4.3, the norm can be lifted to $\beta + \tilde \epsilon/4$. Set $\eta : = \tilde \epsilon/4$. This finishes the proof.

\end{proof}

\begin{lem}
Suppose that $\varphi_1$ is induced from $\varphi_{-}$ such that both $\varphi_{-}$ and $\varphi_1$ have a normalized separation symbol with norm $\beta$ and separation bounds $\delta$. Suppose that $\varphi_1$ is of Fibonacci type and of type $\mathcal B$. Assume that $\varphi_1|U_0^1$ and $\varphi_1|V_0^1$ has the form $h_1 \circ (z-c)^2$ and $h_2 \circ (z-d)^2$.

Set $M :=\delta$. Let $\eta_1$ be given in Lemma 4.11. Assume that the separation symbol of $\varphi_1$ satisfies $s_2^{(1)} \geq \delta$ and $s_4^{(1)} - s_3^{(1)} \leq \eta_1/2$. Then there exists $\eta>0$ depending only on $\delta$ such that $\varphi_3$ has a normalized separation symbol with norm $\beta+\eta$.
\end{lem}

\begin{proof}
Without loss of generality, we may assumt that $\varphi_2$ is also of Fibonacci type and of type $\mathcal C$. Suppose that $\varphi_i$ has a normalized seperation symbol $\Sigma^{(i)} = (s_1^{(i)},s_2^{(i)},s_3^{(i)},s_4^{(i)})$ with norm $\beta$ and corrections $\lambda_1^{(i)}$ and $\lambda_2^{(i)}$, for $i=1,2,3$.

Now for $M := \delta$, let $\eta_1$ and $\eta_2$ be given in Lemma 4.11. We discuss in two cases:
\begin{itemize}
\item[(i)] $\m (U_0^0 \setminus U_P^1) \geq s_3^{(1)} + \eta_1 $;
\item[(ii)] $\partial U_0^1$ is $(M,\eta_2)$-rough.
\end{itemize}

\noindent
{\it Case i.} In this case $\m (U_0^0 \setminus U_P^1) \geq s_4^{(1)} + \eta_1/2 $. Since $\varphi_1 : U_0^1 \setminus U_0^2 \to U_0^0 \setminus U_P^1$ is a degree 2 branched covering, then 
\[\m (U_0^1 \setminus U_0^2) = \frac{1}{2} \m (U_0^0 \setminus U_c^1) \geq \frac{1}{2} s_4^{(1)} + \frac{\eta_1}{4} = s_1^{(2)} + \frac{\eta_1}{4}.
\]
Since $\varphi_1 : A_5^{(2)} \to U_0^1 \setminus U_0^2$ is univalent, $\m A_5^{(2)} \geq s_1^{(2)} + \eta_1/4$. Looking into the proof of Lemma 4.4, we estimate $\m A_5^{(2)}$ as $s_1^{(2)}$. This implies that $\m(U_0^1\setminus U_P^2) = \m (A_3^{(2)}\oplus A_4^{(2)}\oplus A_5^{(2)}) \geq s_4^{(2)} + \eta_1/4$.

\noindent
{\it Case ii.} Since $\partial U_0^1$ is $(M,\eta_2)$-rough, then 
\[ \m(A_2^{(1)} \oplus (U_0^1\setminus U_0^2)) > \m(A_2^{(1)}) + \m (U_0^1\setminus U_0^2)) + \eta_2.
\]
Since $\varphi_1 : A_3^{(2)} \oplus A_5^{(2)} \to A_2^{(1)} \oplus (U_0^1\setminus U_0^2)$ is univalent, we have 
\[ \m (A_3^{(2)} \oplus A_5^{(2)}) > \m(A_2^{(1)}) + \m (U_0^1\setminus U_0^2)) + \eta_2 \geq s_2^{(1)} + s_1^{(2)} + \eta_2 \geq s_3^{(2)} + \eta_2.
\]
Hence $\m(U_0^1\setminus U_P^2) = \m (A_3^{(2)}\oplus A_4^{(2)}\oplus A_5^{(2)}) \geq s_4^{(2)} + \eta_2$.

Use the same argument we can conclude that in either case, 
\[ \min\{\m( A_3^{(2)} \oplus A_4^{(2)} \oplus A_5^{(2)}), \m( B_3^{(2)} \oplus B_4^{(2)} \oplus B_5^{(2)}  )  \} \geq s_4^{(2)} + \epsilon',
\]
where $\epsilon' = \min\{\eta_1/4,\eta_2\}$ depending only on $\delta$. The growth on $s_4^{(2)}$ increases $s_1^{(3)}$ with rate at least $\epsilon'/2$, hence $s_3^{(3)}$ and $s_4^{(3)}$ can be increased by $\epsilon'/2$. By Lemma 4.3, the norm can be lifted to $\beta + \epsilon'/4$. Set $\eta : =\epsilon'/4$. This finishes the proof.
\end{proof}

After all these preparations, we are going to prove Proposition 3. Following \cite{S} (see also \cite{H}[Lemma 6.3]) and by Proposition 4, we can pick $n_f >1$ such that $\phi_n|U^n$ (and $\phi_n|V^n$) has the form $h_n' \circ (z-c)^2$ after conformally conjugate for all $n \geq n_f$. In particular, $h_n'$ has uniformly bounded distortion.

\begin{proof}[Proof of Proposition 3]

Fix $n \geq n_f$, let $\varphi_1$ denote the $n$-th generalized renormalization $\phi_n$. We will consider 7 consecutive steps of inducing $\varphi_i$, $1 \leq i \leq 7$. Assume that $\varphi_i$ has combinatorial type $(p_i, q_i)$. By Lemma 4.4 - 4.6 and Proposition 4, there exists $\beta >0$ depending only on $\mu_1 \geq \tau$ such that $\varphi_i$ has a normalized separation symbol with norm $\beta$ for all $1 \leq i \leq 7$. By Lemma 4.7, each $\varphi_i$ has separation bound $\delta = \beta/4$.

\noindent
{\it Case i.} If there exists $ 1 \leq j \leq 4$ such that $q_j \geq 2$, then corollary 4.8 implies $\varphi_{j+1}$ has a normalized separation symbol with norm $\beta + \delta/8$, so as $\varphi_7$.

\noindent
{\it Case ii.} If there exists $ 1 \leq j \leq 4$ such that $p_j \geq 3$ and $q_j=1$, then Lemma 4.10 implies $\varphi_{j+1}$ has a normalized separation symbol with norm $\beta + \delta/8$, so as $\varphi_7$.

\noindent
{\it Case iii.} If for all  $1 \leq j \leq 4$, $\varphi_j$ is of Fibonacci type. Then there exists $2 \leq k \leq 4$ such that $\varphi_k$ is of type $\mathcal B$. Since $\varphi_{k-1}$ is of Fibonacci type, by Lemma 4.5, $s_2^{(k)} \geq \beta/4$. Set $M := \beta/4$. By Lemma 4.11, 4.12 and 4.13, there exists $\eta>0$ depending only on $\mu_1 \geq \tau$ such that at least $\varphi_{k+3}$ has a normalized separation symbol with norm $\beta + \eta$, so as $\varphi_7$.

As shown above, in any case $\varphi_7$ will have a normalized separation symbol with norm $\beta+\eta$ where $\eta>0$ depending only on $\mu_1 \geq \tau$. Hence the separation norm for $\phi_n$ grows at least as $C' \cdot n$ where $C'$ depending only on $\tau$ and $f$. Then $\mu_n \geq C'/4 \cdot n$ by Proposition 4, which proves Proposition 3.

\end{proof}

\subsection{Absence of Cantor attractor} In this subsection, we will prove Proposition 1 by a random walk argument, using $\alpha_k$ to indicate the state after the $k$-th step in the walk. The Lebesgue measure will be denoted as $m(\cdot)$.

The Schwarzian derivative of a $C^3$ function $\phi: T \to \mathbb R$ is defined as 
\[ S \phi : = \frac{D^3\phi}{D\phi} - \frac{3}{2}(\frac{D^2\phi}{D\phi})^2.
\]
If $f$ is a real polynomial with only real critical points, then $S f <0$ whenever $Df \neq 0$, see for example \cite{MS}[Chapter IV, exercise 1.7]. We shall use the following version of Koebe principle, which was proved in \cite{MS}.

\begin{prop} Assume that $h : T \to h(T)$ is a $C^3$ diffeomorphism with $Sh<0$. If $J$ is a subinterval of $T$ such that $h(J)$ is $\kappa$-well inside $h(T)$, that is $(1 + 2 \kappa)h(J) \subset h(T)$, then $J$ is $\kappa'$-well inside $T$, where $\kappa'=\kappa^2/(1+2\kappa)$.
\end{prop}

\begin{proof}[Proof of Proposition 1]
We start by defining an inducing scheme. To be precise, we may assume that $g_n|I^n \cup C^n \cup J^n \cup D^n$ is of type $\mathcal C$. Since $g_n$ is non-central return, each $I^{n-1}$ and $J^{n-1}$ contains two immediate branches which are symmetric w.r.t. the critical points. Specifying them by $L_n < c < \hat L_n$ and $R_{n} < d < \hat R_{n}$.

Let $W_{n+1} \supset I^{n+1}$ be such that $g_n(\partial W_{n+1}) \subset \partial D^n$, then $W_{n+1} \subset I^n$. Take $J^{n+1} \subset V_{n+1} \subset J^{n}$ analogously. Then $W_{n+1}$ and $V_{n+1}$ are disjoint with the immediate branches. As in the proof of Lemma 2.1 and Lemma 2.2, we can pullback $J^n$ and $D^n$ under $g_n|D^n$ for $r_n-1$ times. Denote $H_0 : = D^n$. For $1 \leq i \leq r_n-1$, denote $K_i: = (g_n|D^n)^{-i}(J^n)$, $H_i: = (g_n|D^n)^{-i}(D^n)$. Then $g_n(c) \in K_{r_n-1}$. Note that it may happen that $g_n(W_{n+1}) \supset H_{r_n-1}$. If $x \in H^i \setminus H^{i+1}$, $0 \leq i \leq r_n-2$, then $g_n^i(x) \in D^n$ but $g_n^{i+1}(x) \notin D^n$. In particular, $g_n^{r_n-1}(x) \in D^n$ iff $x \in H_{r_n-1}$.

Now let us define the induced map $G$ inductively on each level $(I^n \setminus I^{n+1} )\cup (J^n \setminus J^{n+1})$. First, on $(I^0 \setminus I^1) \cup (J^0 \setminus J^1)$, $G$ is just defined as the first return map to $I^0 \cup J^0$. For $n \geq 1$, if $x \in (W_{n+1} \setminus I^{n+1}) \cup (V_{n+1} \setminus J^{n+1})$, define $E(x) := \max\{ i \leq r_n, g_n^i (x) \in C^n \cup D^n\}$. Then define $G$ on $(I^n \setminus I^{n+1} )\cup (J^n \setminus J^{n+1})$ as:
\begin{equation*}\label{eqn:G}
G(x) = \begin{cases}
g_{n}(x) & \text{ if $x \in L_n \cup \hat L_n \cup R_n \cup \hat R_n$ ; }
\\
g_n^{E(x) +1}(x) & \text{ if $x \in (W_{n+1} \setminus I^{n+1}) \cup (V_{n+1} \setminus J^{n+1})$ s.t. $g_n^{E(x) +1}(x) \in I^n \cup J^n$;} 
\\
g_n^{E(x) +2}(x) & \text{ if $x \in (W_{n+1} \setminus I^{n+1}) \cup (V_{n+1} \setminus J^{n+1})$ s.t. $g_n^{E(x) +1}(x) \notin I^n \cup J^n$;} 
\\
g_{n}^2(x) & \text{ otherwise. }
\end{cases}
\end{equation*}
In this manner, $G$ is defined for m-a.e. on level $n$. If $x$ belongs to the immediate branches, then there exists a neighborhood $U_x \ni x$ such that $G$ maps $U_x$ onto $I^n$ or $J^n$ monotonically. If $x$ satisfies the second case, then $U_x$ is the pullback of some $K_{E(x)-1}$ and hence $G$ maps $U_x$ onto $I^n$ or $J^n$. If $x$ satisfies the third case, then $g_n$ maps $U_x$ into $H_{E(x)-2} \setminus (H_{E(x)-1} \cup K_{E(x)-1})$, hence $G$ maps $U_x$ onto $I^{n-1}$ or $J^{n-1}$. For $x$ otherwise, $G$ maps $U_x$ onto $I^{n-1}$ or $J^{n-1}$. Moreover, any two neighborhoods $U_x$ and $U_y$ are either disjoint or coincide.

Repeating this construction for all $n$, we obtain an induced Markov map $G$ which preserves the partition given by the boundary points of the intervals $I^n$ and $J^n$.

By theorem 1, given any $\rho>0$ large enough, there exists $N' >0$ such that the scaling factor $\lambda_n \leq 1/(1+ 2\rho)$ for all $n \geq N'$. By non-flatness and Proposition 5, there exists $\rho'=\mathcal O(\sqrt \rho)$ such that $(1+2\rho') J \subset I^{n-1}$ or $J^{n-1}$ for any return domains $J$ of $g_n$. In particular, $\rho' \to \infty$ and $ \rho \to \infty$. Similarly, there exists $\tilde \rho=\mathcal O(\sqrt \rho')$ such that $(1+2\tilde \rho)W_{n+1} \subset I^n$ and $(1+2\tilde \rho)V_{n+1} \subset J^n$.

To describe the random walk, let $\alpha_k = n$ if $G^k(x) \in (I^n \setminus I^{n+1} )\cup (J^n \setminus J^{n+1})$. The $\alpha_k$ can be considered as random variable with the following conditional probabilities. For $n > N'$ we claim that
\begin{align*}
P(\alpha_{k+1}=n-1|\alpha_k=n)&:= \frac{m(\alpha_k=n \mbox{ and } \alpha_{k+1}=n-1)}{m(\alpha_k=n)}\\
& \geq 1 - \mathcal O(1/\tilde \rho).
\end{align*}
Indeed, if $\alpha_k(x)=n$, then there is a neighborhood $W_x \ni x$ such that $G^k$ maps $W_x$ monotonically onto $I^n$ or $J^n$. Without loss of generality we may assume that $G^k(W_x) = I^n$. Given $y \in W_x$, there are two cases to consider.
\begin{itemize}
\item $G^k(y) \in L_n \cup \hat L_n \cup \bigcup_i J_i$ where $J_i \subset W_{n+1}$ such that $G(J_i) = I^n$ or $J^n$. Then by Proposition 5,
\[ m(\{ y \in W_x | G^k(y) \in  L_n \cup \hat L_n \cup \bigcup_i J_i\}) = \mathcal O(1/\tilde \rho) m(W_x).
\]
\item $G^k(y) \in U$ where $U$ is a domain of $G$ such that $G(U) = I^{n-1}$ or $J^{n-1}$. We may assume that $G(U) = I^{n-1}$. Then by Proposition 5,
\[ m(\{ y \in W_x|G^k(y) \in U, G^{k+1}(y) \in I^n  \}) = \mathcal O(1/\rho) m(W_x \cap G^{-k}(U)).
\]
\end{itemize}
Combining these estimates and adding over all domains $W_x$, we arrive at the claim. A similar argument gives for $r \geq 1$,
\begin{align*}
P(\alpha_{k+1}= n +r| \alpha_k=n)&= \frac{m(\alpha_k=n \mbox{ and } \alpha_{k+1}=n+r)}{m(\alpha_k=n)}\\
& \leq O(\rho^{-r}).
\end{align*}
Therefore, the drift of the random walk is 
\[\mathbb E(\alpha_{k+1}-n|\alpha_k=n) = \sum_{r \geq -1} r P(\alpha_{k+1}= n +r| \alpha_k=n) \leq -\frac{1}{2},
\]
for $\rho$ and hence $\tilde \rho$ sufficiently large and $n > N'$. A similar computation shows that the variance is bounded as well:
\begin{align*}
Var(\alpha_{k+1}-n|\alpha_k=n)& \leq \mathbb E((\alpha_{k+1}-n)^2|\alpha_k=n)\\
& =\sum_{r \geq -1} r^2 P(\alpha_{k+1}= n +r| \alpha_k=n)\\
&< 1+ \sum_{r \geq -1} r^2 O(\rho^{-r}) < \infty.
\end{align*}
Then we can apply the random walk argument from \cite{BKNS} to conclude that $\liminf_{k \to \infty} \alpha_k < \infty$ for m-a.e. $x$, and this excludes the existence of a Cantor attractor.
\end{proof}

\section*{Acknowledgements}

H. Ji is supported by NSFC Grant No. 12301103.


%

\section*{Conflict of interest}

The authors declare that they have no conflict of interest.



\end{document}